\documentclass[11pt]{article}

\usepackage[margin=1in]{geometry}

\pdfoutput=1
\usepackage{amsmath}
\usepackage{amsfonts}
\usepackage{amssymb}
\usepackage{graphicx}
\usepackage[caption=false]{subfig}
\usepackage{amsthm}
\usepackage{mathdots}
\usepackage[usenames,dvipsnames]{xcolor}
\usepackage{tikz}
\usepackage{color}
\usepackage{bm}
\usepackage{hyperref}
\hypersetup{colorlinks=true,
    linkcolor=BrickRed,
    citecolor=OliveGreen,
    urlcolor=MidnightBlue}

\usepackage{booktabs}
\usepackage{algorithmicx}
\usepackage{algpseudocode}
\usepackage{longtable}

\usetikzlibrary{arrows}
\usetikzlibrary{patterns}
\usetikzlibrary{matrix}

\theoremstyle{plain}
\newtheorem{thm}{Theorem}[section]
\newtheorem{proposition}[thm]{Proposition}
\newtheorem{conjecture}[thm]{Conjecture}
\newtheorem{theorem}[thm]{Theorem}
\newtheorem{lemma}[thm]{Lemma}

\newtheorem{question}[thm]{Question}
\newtheorem{corollary}[thm]{Corollary}
\newtheorem{example}[thm]{Example}
\theoremstyle{definition}
\newtheorem{definition}[thm]{Definition}

\newtheorem{remark}[thm]{Remark}

\newcommand{\RR}{\mathbb{R}}
\newcommand{\ZZ}{\mathbb{Z}}
\newcommand{\cS}{\mathcal{S}}

\newcommand{\cE}{\mathcal{E}}
\newcommand{\cC}{\mathcal{C}}
\newcommand{\cG}{\mathcal{G}}
\newcommand{\cA}{\mathcal{A}}

\newcommand{\cV}{\mathcal{V}}

\newcommand{\psd}{\succeq}
\newcommand{\pd}{\succ}

\newcommand{\tr}{\textup{tr}}

\newcommand{\Hyp}{\mathcal{H}}
\newcommand{\SOSHyp}{\mathcal{H}^{\textup{SOS}}}

\title{Certifying polynomial non-negativity\\ via hyperbolic optimization}
\author{James Saunderson\thanks{Department of Electrical and Computer Systems Engineering, Monash University,
VIC 3800, Australia. \texttt{james.saunderson@monash.edu}}}

\begin{document}
\maketitle
\begin{abstract}
We describe a new approach to certifying the global nonnegativity of
multivariate polynomials by solving hyperbolic optimization
problems---a class of convex optimization problems that generalize semidefinite
programs. We show how to produce families of nonnegative polynomials (which we
call hyperbolic certificates of nonnegativity) from any hyperbolic polynomial.
We investigate the pairs $(n,d)$ for which there is a hyperbolic polynomial of
degree $d$ in $n$ variables such that an associated hyperbolic certificate of
nonnegativity is not a sum of squares. If $d\geq 4$ we show that this occurs
whenever $n\geq 4$.  In the degree three case, we find an explicit hyperbolic
cubic in $43$ variables that gives hyperbolic certificates that are not sums of
squares. As a corollary, we obtain the first known hyperbolic cubic no power of which has a
definite determinantal representation. Our approach also allows us to 
show that, given a cubic $p$, and a direction $e$, the decision problem
``Is $p$ hyperbolic with respect to $e$?'' is co-NP hard.
\end{abstract}

\section{Introduction}
\label{sec:intro}

The problem of deciding nonnegativity of a multivariate polynomial is central
to solving optimization and feasibility problems expressed in terms of
polynomials. These, in turn, arise naturally in a wide range of applications
including control systems and robotics, combinatorial optimization, game
theory, and quantum information (see Section~\ref{sec:intro-pop} for further discussion, and
the edited volume~\cite{blekherman2012semidefinite} for an introduction to these 
ideas).

One of the benefits of a polynomial formulation of an optimization problem is
that one can then construct a hierarchy of more tractable `relaxations' of the
problem, based on the fact that being a sum of squares of polynomials is a 
sufficient condition for a polynomial to be nonnegative.
Deciding whether a polynomial is a sum of squares can be formulated as a
semidefinite programming feasibility
problem~\cite{shor1987class,nesterov2000squared,parrilo2003semidefinite,lasserre2001global}.
While such semidefinite programming-based relaxations of semialgebraic problems
have proven very useful for a range of problems, there has been notable recent
progress (both qualitative~\cite{scheiderer2018spectrahedral} and
quantitative~\cite{lee2015lower}) demonstrating the limitations of this approach.

Hyperbolic optimization problems (or hyperbolic programs) are a family of
convex optimization problems that generalize semidefinite optimization
problems~\cite{guler1997hyperbolic}. These involve maximizing a linear
functional over the intersection of an affine subspace and a hyperbolicity cone
(a convex cone constructed from a hyperbolic polynomial, which is a
multivariate polynomial with certain real-rootedness properties that we define
precisely in Section~\ref{sec:prelim}). Hyperbolic programs enjoy many of the
good algorithmic properties of semidefinite
programs~\cite{guler1997hyperbolic}.  Despite this, algorithms for hyperbolic
programming are less mature than those for semidefinite programming.  Indeed,
much of the recent research on hyperbolic programming has focused on trying to
reformulate various classes of hyperbolic programs as semidefinite programs, or on 
related geometric questions associated with the `generalized Lax conjecture'
and its variants~\cite{amini2019spectrahedrality,branden2014hyperbolicity,kummer2015hyperbolic,saunderson2015polynomial,saunderson2018spectrahedral}.
This lack of algorithmic development may be the result of not having generic
ways to produce hyperbolic programming-based formulations and/or relaxations of
polynomial optimization problems, beyond those already captured by
sums of squares. 

This paper introduces ways to construct families of nonnegative polynomials
that can be searched over via hyperbolic optimization. We call these
\emph{hyperbolic certificates of nonnegativity}. By appropriately choosing the
data that specify such a family of nonnegative polynomials, we can recover sums
of squares certificates of nonnegativity. As such, using the ideas in this
paper, we could construct hyperbolic programming-based relaxations of
polynomial feasibility and optimization problems by replacing sums of squares
relaxations with relaxations based on hyperbolic certificates of nonnegativity.
One significant challenge, not addressed in this paper, is that many choices
need to be made to specify a family of hyperbolic certificates of
nonnegativity. Given a specific structured class of polynomial optimization
problems, it is currently unclear which choices might be appropriate to obtain
strong hyperbolic programming-based relaxations using the framework presented
in this paper.

A main focus of the paper is the construction of hyperbolic polynomials for
which the associated hyperbolic certificates of nonnegativity are not sums of
squares.  These are of interest because they have the potential to form the
basis of hyperbolic programming-based relaxations of polynomial optimization
problems that are stronger than semidefinite programming-based relaxations of
comparable complexity.

\subsection{Contributions}

We now discuss the contributions of the paper in more technical detail, and
indicate where the main results appear in the paper.

\paragraph{Hyperbolic certificates of nonnegativity}
Given a hyperbolic polynomial $p$ of degree $d$ in $n$ variables and a
direction of hyperbolicity $e\in \RR^n$ (see Section~\ref{sec:hyp-prelim} for
the definition of these terms), we construct a polynomial map
\[ (x,y) \mapsto \phi_{p,e}(x,y)[\cdot]\]
 from $\RR^{n}\times\RR^d$ to linear functionals on $\RR^n$ such that
$\phi_{p,e}(x,y)[u] \geq 0$ for all $x,y$ if, and only if, $u$ is in the
hyperbolicity cone associated with $p$ and $e$ (Theorem~\ref{thm:main1}). This
slightly extends a related construction due to Kummer, Plaumann, and
Vinzant~\cite{kummer2015hyperbolic}. 

If $f$ and $g$ are polynomial maps, then $\phi_{p,e}(f(z),g(z))[u]$ is a
nonnegative polynomial in $z$ whenever $u$ is in the hyperbolicity cone
associated with $p$ and $e$. If a polynomial can be written this way for a
choice of hyperbolic polynomial $p$, direction of hyperbolicity $e$, and
polynomial maps $f$ and $g$, we say it has a \emph{hyperbolic certificate of
nonnegativity} (see Definition~\ref{def:hypnn}).  We can search for such a
description of a polynomial (for fixed $p$, $e$, $f$, and $g$) by solving a
hyperbolic optimization problem. Moreover, by appropriately specifying these
data, we can recover sums of squares certificates of nonnegativity
(Proposition~\ref{prop:cert-sos}).

\paragraph{Hyperbolic and SOS-hyperbolic polynomials}
If all nonnegative polynomials of the form $\phi_{p,e}(x,y)[u]$ are sums of
squares, we say that $p$ is SOS-hyperbolic with respect to $e$ (see Definition~\ref{def:soshyp}).  
From the point of view of this paper, we are
most interested in hyperbolic polynomials that are \emph{not} SOS-hyperbolic,
since these give rise to tractable families of nonnegative polynomials that go
beyond sums of squares.

In Sections~\ref{sec:sos} and~\ref{sec:cubics} we investigate the degrees $d$,
and numbers of variables $n$, for which there is a hyperbolic polynomial that
is not SOS-hyperbolic.  In Section~\ref{sec:sos} we 
characterize when this happens for $d\geq 4$ by showing that the
specialized V\'amos polynomial,
\[ p(x_1,x_2,x_3,x_4) = x_3^2x_4^2 + 4(x_1x_2x_3+x_1x_2x_4+x_1x_3x_4+x_2x_3x_4)(x_1+x_2+x_3+x_4),\]
is hyperbolic but not SOS-hyperbolic (Proposition~\ref{prop:svamos}), and
showing how to take a hyperbolic but not SOS-hyperbolic polynomial and increase
its number of variables, or its degree, and maintain this property
(Propositions~\ref{prop:dv-increase} and~\ref{prop:nup}).  Variations of the
specialized V\'amos polynomial have been studied by
Br\"and\'en~\cite{branden2011obstructions}, Kummer~\cite{kummer2016note},
Kummer, Plaumann, and Vinzant~\cite{kummer2015hyperbolic}, Burton, Vinzant, and
Youm~\cite{burton2014real}, and Amini and Br\"and\'en~\cite{amini2018non}.
Notably,~\cite[Example 5.11]{kummer2015hyperbolic} shows that a closely related
quartic in five variables is hyperbolic but not
SOS-hyperbolic. 

In Section~\ref{sec:cubics} we study the case $d=3$, and obtain a partial
classification of when hyperbolicity and SOS-hyperbolicity coincide.  We show
how to construct, from a simple graph, a one-parameter family of cubic
polynomials that are hyperbolic if, and only if, the parameter is bounded by the
clique number of the graph. This allows us to establish co-NP-hardness
of deciding hyperbolicity of cubic polynomials (Theorem~\ref{thm:np}). By
choosing $\cG = (\cV,\cE)$ to be the icosahedral graph, (with $12$
vertices and $30$ edges), we obtain a hyperbolic cubic 
\begin{equation}
\label{eq:cubic-intro}
	p(x_0,x,y) = x_0^3 - 3x_0\left(\sum_{i\in \cV}x_i^2 + \sum_{\{i,j\}\in \cE}y_{ij}^2\right) + 
	9\sum_{\{i,j\}\in \cE} x_{i}x_j y_{ij},
\end{equation}
in $43 = 1+12+30$ variables that we show is not SOS-hyperbolic. 

The known cases where all hyperbolic polynomials are SOS-hyperbolic occur
because, if a power of a polynomial has a definite determinantal representation,
then it is SOS-hyperbolic (Proposition~\ref{prop:pdet}).  This is a common
generalization of results due to Kummer, Plaumann, and
Vinzant~\cite{kummer2015hyperbolic}, and Netzer, Plaumann, and
Thom~\cite{netzer2013determinantal}.  As such, any hyperbolic polynomial
that is not SOS-hyperbolic is a hyperbolic polynomial for
which no power has a definite determinantal representation. In particular, the
cubic polynomial~\eqref{eq:cubic-intro} associated with the icosahedral graph,
appears to be the first reported hyperbolic cubic with this property.

\paragraph{Parameterizing the dual cone}
The dual cones of hyperbolicity cones are, in general, not well understood. 
In Section~\ref{sec:dual} we show that the polynomial map $(x,y)\mapsto
\phi_{p,e}(x,y)[\cdot]$ almost (i.e., up to closure) parameterizes the dual of the associated hyperbolicity cone
(Theorem~\ref{thm:dualma}). We then  discuss
specific situations in which the map exactly parameterizes the closed dual
cone, giving new `hidden convexity' results, where the
image of a polynomial map is convex.

\subsection{Applications via polynomial optimization}
\label{sec:intro-pop}

We now discuss, in a little more detail, the connection between certifying nonnegativity of polynomials and 
polynomial optimization, and touch on some of the myriad applications of polynomial optimization. 

Polynomial optimization problems involve the minimization of a multivariate polynomial over a set 
defined by polynomial inequality and equality constraints.
By introducing additional variables, one can reduce to the case of minimization over the real points of an algebraic 
variety $\mathcal{V}$ defined by polynomial equations, i.e., 
\begin{equation}
	\label{eq:pop}
	 \textup{minimize}_z\;\; q_0(z)\;\;\textup{subject to}\;\; z\in \mathcal{V}\cap \RR^n.
\end{equation}
For instance, if $q_0(z) = \|\mathcal{A}(z)-y\|^2$, for a linear map $\mathcal{A}$, 
then~\eqref{eq:pop} includes the problem of computing the nearest point to a 
given variety~\cite{draisma2016euclidean}, or more generally, the distance from an affine space to a variety. 
Such problems can often be interpreted as structured linear inverse problems, in which 
we seek to estimate a structured object $z_*\in \mathcal{V}\cap \RR^n$ (such as a rank one tensor) 
from noisy linear measurements, represented by $y$. Problems as diverse as noisy low-rank tensor completion~\cite{barak2016noisy} and 
simultaneous localization and mapping (SLAM) in robotics~\cite{rosen2015convex} can be expressed in this form.
Other cases of~\eqref{eq:pop}, in which $\mathcal{V}\cap \RR \subseteq \{-1,1\}^n$, capture combinatorial optimization problems. 
For instance, in error correction coding, the problem of decoding binary linear codes can be formulated as the optimization of 
a linear functional $q_0$ over the real variety defined by the codewords~\cite{feldman2005using,gouveia2012new}. 

The polynomial optimization problem~\eqref{eq:pop} can be reformulated as the maximization of a lower bound on $q_0$ over
the real points of the variety, i.e., 
\[ \textup{maximize}_{\gamma}\;\; \gamma\;\;\textup{subject to}\;\;  
		q_0(z) - \gamma \geq 0 \;\;\textup{for all $z\in \mathcal{V}\cap \RR^n$}.\]
This leads us to the connection with certificates of nonnegativity. If we can find a globally nonnegative polynomial $q$ such that
$q_0(z) - \gamma$ agrees with $q$ on $\mathcal{V}\cap \RR^n$, it follows that $\gamma$ is a lower bound on the objective 
value of~\eqref{eq:pop}. We can make this computationally tractable by searching over families of 
nonnegative polynomials $q$ with tractable certificates of nonnegativity, such as sums of squares of degree at most $2d$.
This approach gives rise to semidefinite programming relaxations of polynomial optimization problems.  

Hyperbolic certificates of nonnegativity, introduced in this
paper, expand the scope of tractable relaxations possible for polynomial optimization problems. 
If we fix a hyperbolic polynomial $p$, a direction 
of hyperbolicity $e$, and polynomial maps $f$ and $g$, we can obtain a hyperbolic programming relaxation 
of~\eqref{eq:pop} as 
\begin{equation}
\label{eq:hyp-pop}
\textup{maximize}_{\gamma, u}\;\; \gamma\;\;\textup{subject to}\;\; \begin{cases} 
		q_0(z) - \gamma = \phi_{p,e}(f(z),g(z))[u] \;\;\textup{for all $z\in \mathcal{V}$}\\
		u\in \textup{hyperbolicity cone associated with $p$ and $e$}.\end{cases}
\end{equation}
Depending how how $\mathcal{V}$ is represented, there are different approaches to making 
the equality constraints, which are linear in $\gamma$ and $u$, explicit. If it is straightforward
to sample from the irreducible components of $\mathcal{V}$, the sampling-based approach of 
Cifuentes and Parrilo~\cite{cifuentes2017sampling} could be used.
Otherwise one can use methods based on Gr\"obner bases~\cite{permenter2012selecting}.  
We show, in Proposition~\ref{prop:cert-sos}, that there are ways to choose $p$, $e$, $f$, and $g$ 
so that~\eqref{eq:hyp-pop} recovers a sum-of-squares relaxation.
However, it may be possible to make other choices that are more tailored to the structure of the original polynomial
optimization problem of interest. Such targeted choices may have advantages, such as giving
relaxations that can be solved more efficiently than a semidefinite programming-based relaxation, 
and yet give comparable bounds on the optimal value.

\section{Preliminaries}
\label{sec:prelim}

\subsection{Basic notation}
Let $\RR[x_1,x_2,\ldots,x_n]_d$ denote polynomials with real coefficients,
homogeneous of degree $d$ in the indeterminates $x_1,x_2,\ldots,x_n$.  Let
$\cS^m$ denote real symmetric $m\times m$ matrices and let
$\cS[x_1,x_2,\ldots,x_n]^m_{d}$ denote symmetric matrices with entries in
$\RR[x_1,x_2,\ldots,x_n]_{d}$.  If $p\in \RR[x_1,\ldots,x_n]_d$ and $u\in
\RR^n$, then $D_up(x) = \left.\frac{d}{dt}p(x+tu)\right|_{t=0}$ is the
directional derivative of $p$ in the direction $u$. For brevity we write
$D^2_{uv}p(x):= D_u D_v p(x)$.

If $A\in \cS^m$ is a symmetric matrix, we write $A \psd 0$ to mean that $A$ is
positive semidefinite.  We use $\RR^n$ to denote the $n$-dimensional real vector
space and $(\RR^n)^*$ to denote the dual space of linear functionals on
$\RR^n$. If $\xi\in (\RR^n)^*$ and $u\in \RR^n$, we use the notation $\xi[u]$
for the image of $u$ under $\xi$.  Throughout, we let $e_1,e_2,\ldots,e_n\in
\RR^n$ denote the standard basis vectors, so that $e_i$ is zero except for the
$i$th entry, which is one. Occasionally it is convenient to use
$e_0,e_1,\ldots,e_{n}$ as the standard basis for $\RR^{n+1}$.  If $x\in \RR^n$,
we use $\|x\| := \sqrt{x_1^2 + \cdots + x_n^2}$ to denote the Euclidean norm. 

\subsection{Hyperbolic polynomials, hyperbolic eigenvalues, and hyperbolicity cones}
\label{sec:hyp-prelim}

A homogeneous polynomial $p\in \RR[x_1,\ldots,x_n]_d$ is \emph{hyperbolic with respect to
$e\in \RR^n$} if $p(e) > 0$
and, for all $x\in \RR^n$, the univariate polynomial $p_x(t):=p(x+te)$ has only
real zeros.  Throughout we let $\Hyp_{n,d}(e)$ denote the set of polynomials
homogeneous in $n$ variables of degree $d$ that are hyperbolic with respect to
$e\in \RR^n$. 
If $p\in \Hyp_{n,d}(e)$ and $x\in \RR^n$, we denote by $\lambda_1(x)\geq
\lambda_2(x)\geq \cdots \geq \lambda_d(x)$ the zeros of $t\mapsto p(te-x)$, and
often refer to these as the \emph{hyperbolic eigenvalues} of $x$.  These depend
on the choice of $e$, but we will usually suppress this in our notation.
Associated with a hyperbolic polynomial is the closed hyperbolicity cone \[
\Lambda_+(p,e) = \{x\in \RR^n\;:\; \lambda_{i}(x) \geq 0\;\;\textup{for all
$i=1,2,\ldots,d$}\}.\] This is a convex cone, a result due to
G\r{a}rding~\cite{gaarding1959inequality}.  A hyperbolic polynomial $p\in
\Hyp_{n,d}(e)$ is \emph{complete} if $\{x:\lambda_1(x) = \cdots = \lambda_d(x)
= 0\} = \{0\}$.  The hyperbolicity cones of complete hyperbolic polynomials are
\emph{pointed}, in the sense that $\Lambda_+(p,e) \cap (-\Lambda_+(p,e)) =
\{0\}$. 

The following result describes how the eigenvalue functions change, and appears
in a number of slightly different formulations in the
literature~\cite{guler1997hyperbolic,harvey2013gaarding,atiyah1970lacunas}.
Note that the functions $t_i(s;x,u)$ appearing below are just a particular
choice of ordering of the eigenvalues of $x+su$.
\begin{theorem}[{\cite[Lemma 3.27]{atiyah1970lacunas}}]
	\label{thm:abg}
	If $p\in \Hyp_{n,d}(e)$ and $x,u\in \RR^n$, then 
	\[ p(x+te+su) = p(e)\prod_{i=1}^{d}(t+t_i(s;x,u))\]
	where the functions $s\mapsto t_i(s;x,u)$ are real analytic functions of $s$ with the property that
	if $u\in \Lambda_+(p,e)$, then $t_i'(s;x,u) \geq 0$ for all $s$. 
\end{theorem}
The fact that $t_i'(s;x,u)\geq 0$  whenever $u$ is in
the hyperbolicity cone is the key property from which essentially all
nonnegativity statements in this paper can be derived. 

\subsection{Hyperbolic programming}

If $p$ is hyperbolic with respect to $e$, and $\xi$ is a linear functional, a
convex optimization problem of the form \[ \textup{minimize}_x\; \xi[x]
\quad\textup{subject to}\quad \begin{cases} Ax = b&\\x\in
\Lambda_+(p,e)&\end{cases}\] is known as a \emph{hyperbolic optimization
problem}. G\"uler~\cite{guler1997hyperbolic} showed that 
$-\log_ep(x)$ is a self-concordant barrier function for the cone
$\Lambda_+(p,e)$. As such, hyperbolic optimization problems can be solved using
interior point methods as long as the polynomial $p$ can be evaluated
efficiently.  More recently, other algorithmic approaches to solving hyperbolic
optimization problems have been developed, including primal-dual interior point
methods~\cite{myklebust2014interior}, affine scaling
methods~\cite{renegar2014polynomial}, first-order methods based on applying a
subgradient method to a transformation of the
problem~\cite{renegar2016efficient} and accelerated modifications tailored for
hyperbolic programs~\cite{renegar2017accelerated}. 

\subsection{Definite determinantal representations}

If $A_1,\ldots,A_n$ are $d\times d$ symmetric matrices 
and $e\in \RR^n$ satisfies $\sum_{i=1}^{n}A_ie_i \pd 0$, then the polynomial
\begin{equation}
	\label{eq:defdet} \textstyle{p(x) = \det\left(\sum_{i=1}^{n}A_ix_i\right)}
\end{equation}
is homogeneous of degree $d$ and is hyperbolic with respect to $e$.  We say
that a polynomial $p\in \Hyp_{n,d}(e)$ has a \emph{definite determinantal
representation} if it can be expressed in the form~\eqref{eq:defdet} for some
$d\times d$ symmetric matrices $A_1,A_2,\ldots,A_n$ such that
$\sum_{i=1}^{n}A_ie_i \pd 0$. 
In this case 
its hyperbolicity cone is a \emph{spectrahedron}, and has the form
$\Lambda_+(p,e) = \left\{x\in \RR^n\;:\; \sum_{i=1}^{n}A_i x_i \psd 0\right\}$.

\subsection{B\'ezoutians and Hankel matrices}

The nonnegative polynomials we construct from hyperbolic polynomials will come
from the positive semidefiniteness of B\'ezoutians of certain pairs of
polynomials, or of Hankel matrices associated with certain rational functions.
In this section we summarize some basic facts about these objects and the
relationships between them. These can be found, for instance,
in~\cite{bini2011numerical} or~\cite[Section 2.1]{krein1981method}. We give proofs of some of these results 
in Appendix~\ref{app:prelim} to make the paper more self-contained.

If $a(t)$ and $b(t)$ are univariate polynomials with $\textup{deg}(b) < \textup{deg}(a)$, define
the $m\times m$ Hankel matrix
\[ H_m\!\left(\frac{b}{a}\right) = [h_{i+j-1}]_{i,j=1}^{m}\quad\textup{where}\quad
\frac{b(t)}{a(t)} = \sum_{k=1}^{\infty}h_kt^{-k}.\]

If $a(t)$ and $b(t)$ are univariate polynomials with $\textup{deg}(b) < \textup{deg}(a) \leq  m$, 
the B\'ezoutian $B_m(a,b)$ is the $m\times m$ matrix defined via the identity
\begin{equation}
	\label{eq:beqid} \frac{a(t)b(s) - b(t)a(s)}{t-s} = \sum_{i,j=1}^{m} [B_m(a,b)]_{ij}t^{i-1}s^{j-1}.
\end{equation}
If $m> d= \textup{deg}(a)$, then the $m\times m$ B\'ezoutian is zero except in the upper left $d\times d$ block.

Under appropriate assumptions on $a$ and $b$, the B\'ezoutian  and the
Hankel matrix are related by a unimodular congruence transformation.
\begin{proposition}
\label{prop:cong-uni}
Suppose $a$ is a monic polynomial of degree $d$ and $b$ is a polynomial of degree at most $d-1$. If $m\geq d$, 
there exists an $m\times m$ unimodular matrix $M_m$ with entries that are linear in the coefficients of $a$, such that 
$B_{m}(a,b) = M_m(a)H_m\!\left(\frac{b}{a}\right)M_m(a)^T$.
\end{proposition}
\begin{proof}
See Appendix~\ref{app:prelim}.
\end{proof}
Certain linear transformations on polynomials give rise to particularly nice congruence transformations on B\'ezoutians. 
\begin{lemma}
\label{lem:shift}
Let $a$ and $b$ be univariate polynomials of degree at most $d\leq m$. 
Let $(t_0\cdot a)(t) = a(t+t_0)$ and $(t_0\cdot b)(t) = b(t+t_0)$ be shifted versions of 
those polynomials. Then 
\[ B_m(t_0\cdot a,t_0\cdot b) = K(t_0)B_m(a,b)K(t_0)^T\quad\textup{where}\quad [K(t_0)]_{jk} = \binom{k-1}{j-1}t_0^{k-j}\]
with the convention that $\binom{k}{j} = 0$ if $k<j$.
\end{lemma}
\begin{proof}
Combine Theorem 2.7 of~\cite{bini2011numerical} with the discussion on page 34 of~\cite{bini2011numerical}.
\end{proof}

\section{Hyperbolicity cones as sections of nonnegative polynomials}
\label{sec:cert}

In this section we show how to construct, from a hyperbolic polynomial $p\in \Hyp_{n,d}(e)$, 
a subspace of polynomials for which the cone of nonnegative polynomials in the subspace
is linearly isomorphic to the hyperbolicity cone $\Lambda_+(p,e)$. Consequently, we can 
optimize over nonnegative polynomials from this subspace by solving hyperbolic optimization problems.

Our approach is closely related to the following result of Kummer, Plaumann, and Vinzant.
\begin{theorem}[{Kummer, Plaumann, Vinzant~\cite{kummer2015hyperbolic}}]
\label{thm:kpv}
If $p\in \Hyp_{n,d}(e)$ is square-free, then $u\in \Lambda_+(p,e)$ if, and only if, 
$D_up(x)D_ep(x) - p(x)D_{ue}p(x) \geq 0$ for all $x$.
\end{theorem}
This shows that the hyperbolicity cone $\Lambda_+(p,e)$ is linearly isomorphic to the intersection 
of the cone of nonnegative polynomials in $n$ variables of degree $2d-2$ with the subspace spanned by the polynomials
$D_{e_i}p(x)D_ep(x) - p(x)D_{e_ie}p(x)$ for $i=1,2,\ldots,n$.

Our variation on Theorem~\ref{thm:kpv} is  expressed in terms of the 
B\'ezoutian, or alternatively the corresponding Hankel matrix, associated with 
a polynomial and its directional derivative.
\begin{definition}
\label{def:BH}
If $p\in \RR[x_1,\ldots,x_n]$ and $u\in \RR^n$, 
let $p_x(t) = p(x+te)$ and let $D_up_x(t) = D_up(x+te)$. 
The \emph{parameterized B\'ezoutian} is the $d\times d$ symmetric matrix with polynomial entries given by 
\[ B_{p,e}(x)[u] = B_d(p_x,D_up_x).\]
The \emph{parameterized Hermite matrix} is the $d\times d$ symmetric Hankel matrix with polynomial entries given by
\[ H_{p,e}(x)[u] = H_d\!\left(\frac{D_up_x}{p_x}\right).\]
\end{definition}
Note that $B_{p,e}(x)[u]$ and $H_{p,e}(x)[u]$ are both linear in $u$. 
Moreover,  
\[ \frac{D_ep_x(t)}{p_x(t)} = \sum_{k=1}^{d}\frac{1}{t+\lambda_k(x)}\quad\textup{and so}\quad
[H_{p,e}(x)[e]]_{ij} = \sum_{k=1}^{d}(-\lambda_k(x))^{i+j-2}.\]
This is (up to a choice of sign) exactly the parameterized Hermite matrix from~\cite{netzer2013determinantal}.
\begin{example}[Parameterized Hermite matrix for the determinant]
\label{eg:hermdet}
If $p(X) = \det(X)$ is the determinant restricted to symmetric matrices, $e = I$, and $U$ is a symmetric matrix, then 
\[ H_{p,e}(X)[U]_{ij} = [\tr(UX^{i+j-2})]_{ij}.\]
This follows from the fact that
\[ \frac{D_U\det(X+tI)}{\det(X+tI)} = D_U \log\det(X+tI)=   \tr(U(X+tI)^{-1}) = \sum_{k=1}^{\infty} \tr(U(-X)^{k-1})t^{-k}.\] 
\end{example}
The following relationship between the 
parameterized B\'ezoutian and Hermite matrices  
allows positivity statements about parameterized B\'ezoutians to be transferred to 
corresponding positivity statements about parameterized Hermite matrices, and vice versa. 
\begin{proposition}
\label{prop:congx}
If $p\in \RR[x_1,\ldots,x_n]_d$ and $u\in \RR^n$, then there are matrices $M_{p,e}(x)$
and $M_{p,e}(x)^{-1}$, both with polynomial entries, such that   
$B_{p,e}(x)[u] = M_{p,e}(x)H_{p,e}(x)[u]M_{p,e}(x)^T$.
\end{proposition}
\begin{proof}
First assume that $p(e)=1$ so that $p_x(t)$ is monic. The result then follows from Proposition~\ref{prop:cong-uni}
and the fact that the coefficients of $p_x(t)$ are polynomials in $x$. Moreover, in this case $M_{p,e}(x)$ has determinant one.
For the general case, write $p(x) = p(e)\tilde{p}(x)$ where $\tilde{p}_x(t)$ is monic. 
Then $B_{p,e}(x)[u] = p(e)^2B_{\tilde{p},e}(x)[u]$ and $H_{p,e}(x)[u] = H_{\tilde{p},e}(x)[u]$ so we 
have $M_{p,e}(x) = p(e)M_{\tilde{p},e}(x)$ from which we see that 
$M_{p,e}(x)^{-1} = p(e)^{-1}M_{\tilde{p},e}(x)^{-1}$ which has polynomial entries. 
\end{proof}
The following result (the essence of which goes back to Hermite), 
gives a characterization of hyperbolic polynomials in terms of the parameterized Hermite matrix.
A statement in this form can be found, for instance, in~\cite{netzer2013determinantal}.
\begin{theorem}
\label{thm:npt}
Given $p\in \RR[x_1,\ldots,x_n]_d$ and $e\in \RR^n$, we have that $p\in \Hyp_{n,d}(e)$ if, and only if, $H_{p,e}(x)[e]\psd 0$ for all $x\in \RR^n$. 
\end{theorem}
By using Proposition~\ref{prop:congx}, this characterization of hyperbolicity can also be expressed in 
terms of the parameterized B\'ezoutian.
\begin{corollary}
\label{cor:npt}
Given $p\in \RR[x_1,\ldots,x_n]_d$ and $e\in \RR^n$, we have that $p\in \Hyp_{n,d}(e)$ if, and only if, $B_{p,e}(x)[e]\psd 0$ for all $x\in \RR^n$. 
\end{corollary}
Our main result for this section shows that these tests for hyperbolicity can be extended to give a
description of the full hyperbolicity cone. We defer the proof until Section~\ref{sec:pf}. 
\begin{theorem}
	\label{thm:main1}
	If $p\in \Hyp_{n,d}(e)$, then 
	\begin{align*}
	\Lambda_+(p,e) & = \{u\in \RR^n\;:\; H_{p,e}(x)[u] \psd 0,\;\;\textup{for all $x\in \RR^n$}\}\\
	& = \{u\in \RR^n\;:\; B_{p,e}(x)[u] \psd 0,\;\;\textup{for all $x\in \RR^n$}\}.
	\end{align*}
\end{theorem}
\begin{remark}
 The $(1,1)$ entry of $B_{p,e}(x)[u]$ is $D_up(x)D_ep(x) - p(x)D_{ue}p(x)$. It follows from Theorem~\ref{thm:main1} that, 
if $p\in \Hyp_{n,d}(e)$, and  
$u\in \Lambda_+(p,e)$, then $[B_{p,e}(x)[u]]_{11} \geq 0$. This is one direction of Kummer, Plaumann, and 
Vinzant's result (Theorem~\ref{thm:kpv}). 
\end{remark}
In what follows, it is sometimes convenient to use the following variations on Theorems~\ref{thm:npt} and~\ref{thm:main1}, respectively.
In some arguments they allow us to reduce the number of variables in certain polynomials by one.
\begin{corollary}
\label{cor:npt2}
Let $p\in \RR[x_1,\ldots,x_n]_d$, $e\in \RR^n$, and let $W\subseteq \RR^n$ be a codimension one subspace such that 
$e\notin W$. Then $p\in \Hyp_{n,d}(e)$ if, and only if, $B_{p,e}(x)[e]\psd 0$ for all $x\in W$,
which holds if, and only if, $H_{p,e}(x)[e] \psd 0$ for all $x\in W$. 
\end{corollary}
\begin{corollary}
	\label{cor:main2}
	If $p\in \Hyp_{n,d}(e)$ and $W\subseteq \RR^n$ is a subspace such that $\textup{codim}(W) =1$ and $e\notin W$, then 
	\begin{align*}
	\Lambda_+(p,e) & = \{u\in \RR^n\;:\; H_{p,e}(x)[u] \psd 0,\;\;\textup{for all $x\in W$}\}\\
	& = \{u\in \RR^n\;:\; B_{p,e}(x)[u] \psd 0,\;\;\textup{for all $x\in W$}\}.
	\end{align*}
\end{corollary}
Corollaries~\ref{cor:npt2} and~\ref{cor:main2} follow from the following observation about parameterized
B\'ezoutians. 
\begin{proposition}
\label{prop:eshift}
Suppose that $p\in \RR[x_1,\ldots,x_n]_d$ and $u\in \RR^n$. 
Then there exists a unimodular polynomial matrix $K(t_0)$ such that 
$B_{p,e}(x+t_0e)[u] = K(t_0)B_{p,e}(x)[u]K(t_0)^T$ for all $x\in \RR^n$ and all $t_0\in \RR$.
\end{proposition}
\begin{proof}
This follows directly from Lemma~\ref{lem:shift} since
\[ B_{p,e}(x+t_0e)[u] = B_d(t_0\cdot p_x,t_0\cdot D_up_x) =K(t_0)B_d(p_x,D_up_x)K(t_0)^T\]
for all $x\in \RR^n$ and all $t_0\in \RR$. Furthermore, one can directly check from Lemma~\ref{lem:shift} that $K(t_0)$ 
is upper triangular and $\det(K(t_0)) = 1$.  
\end{proof}
It follows immediately that $B_{p,e}(x)[u] \psd 0$ for all $x\in \RR^n$ if, and only if, $B_{p,e}(x)[u] \psd 0$ for all $x\in W$
whenever $W$ is a codimension one subspace of $\RR^n$ and $e\notin W$. Corollary~\ref{cor:npt2} then follows from Theorem~\ref{thm:npt} and Corollary~\ref{cor:npt}. 
Similarly Corollary~\ref{cor:main2} then follows from Theorem~\ref{thm:main1}.

\subsection{Hyperbolic certificates of nonnegativity}

One consequence of Theorem~\ref{thm:main1} is that if $p\in \Hyp_{n,d}(e)$ and $u\in \Lambda_+(p,e)$, 
then both of the following polynomials
\begin{align}
	\phi_{p,e}^B(x,y)[u] &:= y^TB_{p,e}(x)[u]y\quad\textup{and}\label{eq:phiB}\\
	\phi_{p,e}^H(x,y)[u] &:= y^TH_{p,e}(x)[u]y,\label{eq:phiH}
\end{align}
are globally nonnegative in $x$ and $y$. By composing the polynomials $\phi_{p,e}^B(x,y)[u]$ or $\phi_{p,e}^H(x,y)[u]$ 
with other polynomial maps we obtain further nonnegative polynomials. 
\begin{definition}
\label{def:hypnn}
We say that a polynomial $q$ in $m$ variables has a \emph{hyperbolic certificate of nonnegativity with respect to $(p,e)$}
if there exists $u\in \Lambda_+(p,e)$ and polynomial maps $f:\RR^m\rightarrow \RR^n$ and $g:\RR^m\rightarrow \RR^d$ such that 
\begin{equation}
	\label{eq:hypcert} q(z) = \phi_{p,e}^H(f(z),g(z))[u]\quad\textup{for all $z\in \RR^m$}.
\end{equation}
\end{definition}
Since the parameterized B\'ezoutian and Hermite matrix are the same up to a unimodular congruence 
transformation (see Proposition~\ref{prop:congx}), 
there is no difference between using $\phi_{p,e}^B$ or $\phi_{p,e}^H$ in Definition~\ref{def:hypnn}. We can transform 
from one representation to another by changing $g$ appropriately. From now on, we will often write $\phi_{p,e}$ instead
of $\phi_{p,e}^H$ unless we specifically want to work with the B\'ezoutian formulation.

Any polynomial that has a hyperbolic certificate of nonnegativity is nonnegative due to Theorem~\ref{thm:main1}.
Moreover, given a polynomial $q\in \RR[x_1,\ldots,x_m]_{2d}$, the problem of searching for a hyperbolic certificate of nonnegativity
of $q$ can be cast as a hyperbolic feasibility problem:
\[ \textup{find}\;\; u\in \Lambda_+(p,e)\;\; \textup{such that}\;\; q(z) = \phi_{p,e}(f(z),g(z))[u]\quad\textup{for all $z\in \RR^m$}\]
which aims to find a point in the intersection of the hyperbolicity cone and the affine subspace defined by, for instance,
equating coefficients in the polynomial identity~\eqref{eq:hypcert}.

\paragraph{Recovering sums of squares certificates}
A homogeneous polynomial $q\in \RR[x_1,\ldots,x_m]_{2d}$ is a \emph{sum of squares} if there
is a positive integer $k$ and homogeneous polynomials $p_1,\ldots,p_k\in \RR[x_1,\ldots,x_m]_{d}$ such that 
$q(x) = \sum_{i=1}^{k}p_i(x)^2$ for all $x\in \RR^m$. Clearly any sum of squares is nonnegative. Furthermore, it is well known that
if $m_d(x)$ is the vector of all monomials that are homogeneous of degree $d$ in $m$ variables then 
$q$ is a sum of squares if, and only if,
\[ \textup{there exists}\quad Q\in \cS_{+}^{\binom{m+d-1}{d}}\quad\textup{such that}\quad q(x) = m_d(x)^TQm_d(x).\]
This allows one to search for a sum of squares certificate of the nonnegativity of $q$ via solving a semidefinite feasibility problem. 

If $q$ is a sum of squares, we can choose the data ($p$, $e$, $f$, and $g$) in Definition~\ref{def:hypnn} to 
give a hyperbolic certificate of nonnegativity for $q$. This shows that our notion of hyperbolic certificates of 
nonnegativity captures sums of squares as a special case. 

\begin{proposition}
\label{prop:cert-sos}
Let $q\in \RR[x_1,\ldots,x_m]_{2d}$ be a sum of squares, and let 
$Q \in \mathcal{S}^{\binom{m+d-1}{d}}_+$ be such that $q(x) = m_d(x)^TQm_d(x)$.
Then $q$ has a hyperbolic certificate of nonnegativity as
\[ q(x) = \phi_{p,e}(F(x),g(x))[U]\]
where  
\[ p(X) = \det(X),\quad e = I,\quad U = \begin{bmatrix} 0 & 0\\0 & Q\end{bmatrix},\;\; 
F(x) = \begin{bmatrix} 0 & m_d(x)^T\\m_d(x) & 0\end{bmatrix}\quad \textup{and} \quad g(x) = e_2.\]
\end{proposition}
\begin{proof}
From our choice of $p$ and $g$, we see that 
$\phi_{p,e}(F(x),g(x))[U] = e_2^TH_{\det,I}(F(x))[U]e_2$.
From~\eqref{eg:hermdet},
\begin{align*}
	 e_2^TH_{\det,I}(F(x))[U]e_2 = \tr(U F(x)^2) & = 
\tr\left(\begin{bmatrix} 0 & 0\\0 & Q\end{bmatrix}\begin{bmatrix} m_d(x)^Tm_d(x) & 0\\0 & m_d(x)m_d(x)^T\end{bmatrix}\right)\\
	& =  m_d(x)^TQm_d(x) = q(x).
\end{align*}
\end{proof}
We have now seen that every sum of squares has a hyperbolic certificate of nonnegativity. 
In Section~\ref{sec:sos} we will show that there are polynomials that have hyperbolic certificates of nonnegativity, but 
that are not sums of squares.
 
\section{Hyperbolic certificates and sums of squares}
\label{sec:sos}
In this section we study conditions under which polynomials with hyperbolic certificates of nonnegativity are, or are not, 
sums of squares. We will often phrase this in terms of sums of squares certificates of the positive semidefiniteness 
of the parameterized B\'ezoutian (and Hermite matrix), which are matrices with polynomial entries. 
\begin{definition}
	A $d\times d$ symmetric matrix $P\in \cS[x_1,\ldots,x_n]^d$ with polynomial entries is a 
	\emph{matrix sum of squares} if there exists a positive integer $\ell$ and a $d\times \ell$ matrix $Q$ with polynomial
	entries such that $P(x) = Q(x)Q(x)^T$.
\end{definition}
It is well known that $P(x)$ is a matrix sum of squares if and only if the polynomial $y^TP(x)y$ is a sum of squares in $x$ and $y$. 
We will freely pass between these two equivalent definitions. 

The following is the central definition of this section and Section~\ref{sec:cubics}. 
It specifies a class of hyperbolic polynomials for which we have not just a sum of 
squares certificate of their hyperbolicity, but also a sum of squares description of
the hyperbolicity cone.  
\begin{definition}
\label{def:soshyp}
If $p\in \Hyp_{n,d}(e)$, we say that $p$ is \emph{SOS-hyperbolic with respect to $e$} if $B_{p,e}(x)[u]$ is a matrix sum of 
squares for all $u\in \Lambda_+(p,e)$. 
\end{definition}
We use the shorthand notation $\SOSHyp_{n,d}(e)\subseteq \Hyp_{n,d}(e)$ for the collection of polynomials 
that are homogeneous of degree $d$ in $n$ variables and SOS-hyperbolic with respect to $e$. If $p\in \SOSHyp_{n,d}(e)$, 
and $q$ has a hyperbolic certificate of nonnegativity via an identity of the form $q(z) = \phi_{p,e}(f(z),g(z))$, then 
$q$ is a sum of squares. As such, we are most interested in hyperbolic polynomials that are \emph{not} SOS-hyperbolic, 
since these give  new certificates of nonnegativity. 

Table~\ref{tab:nnh} describes our understanding of the values of $n$ and $d$ for which 
the sets $\SOSHyp_{n,d}(e)$ and $\Hyp_{n,d}(e)$ coincide. In this section we develop some preparatory results, 
establish the equality cases of the table, and the remainder of the first and third columns of the table 
(corresponding to $d=2$ and $d\geq 4$). 
In Section~\ref{sec:cubics} we focus on the $d=3$ case.

\begin{table}
{\footnotesize
\caption{The $(n,d)$ entry of the table is $=$ if $\Hyp_{n,d}(e) = \SOSHyp_{n,d}(e)$ for all $e\in \RR^n$. 
The $(n,d)$ entry of the table is 
$\neq$ if $\Hyp_{n,d}(e) \neq \SOSHyp_{n,d}(e)$ for some $e$. The entry is 
`?' in the cases that are yet to be resolved. The entries in bold are new results
(with the exception of $d=4$ and $5\leq n \leq 8$ which follow from~\cite[Example 5.11]{kummer2015hyperbolic}).}
\label{tab:nnh} 
\begin{center}
\begin{tabular}{c|ccc}
$n\setminus d$ & $d=2$ & $d=3$ & $d\geq 4$\\
\hline
$n=3$ & $=$ & $=$ & $=$ \\
$n=4$ & $=$ & $=$ & $\bm{\neq} $\\
$5\leq n \leq 42$ & $=$ & ? & $\bm{\neq}$\\
$n\geq 43$ & $=$ & $\bm{\neq}$ & $\bm{\neq}$
\end{tabular}
\end{center}
}
\end{table}

Before giving proofs establishing the entries of the table, we give a number of equivalent characterizations
for polynomials that are SOS-hyperbolic with respect to some direction $e$. 
\begin{proposition}
\label{prop:sos-equiv}
Suppose that $p\in \SOSHyp_{n,d}(e)$ and let $p_x(t):=p(x+te)$ and $D_up_x(t) = D_up(x+te)$. Let $W$ be an $n\times (n-1)$
matrix such that the $n\times n$ matrix  $\begin{bmatrix} W & e\end{bmatrix}$ has full rank. Then the following are equivalent
\begin{enumerate}
	\item $u\in \Lambda_+(p,e)$
	\item $B_{p,e}(x)[u]\in \cS[x_1,\ldots,x_n]^d$ is a matrix sum of squares
	\item $B_{p,e}(Wz)[u]\in \cS[z_1,\ldots,z_{n-1}]^d$ is a matrix sum of squares
	\item $H_{p,e}(x)[u]\in \cS[x_1,\ldots,x_n]^d$ is a matrix sum of squares 
	\item $H_{p,e}(Wz)[u]\in \cS[z_1,\ldots,z_{n-1}]^d$ is a matrix sum of squares 
\end{enumerate}
\end{proposition}
\begin{proof}
 The equivalence of 1 and 2 is just the definition of $p$ being SOS-hyperbolic with respect to $e$. 
The equivalence of 2 and 4 and of 3 and 5 both follow from Proposition~\ref{prop:congx}. The equivalence of 2 and 3 follows 
from Proposition~\ref{prop:eshift}.
\end{proof}
\begin{remark}
	If $p\in \SOSHyp_{n,d}(e)$, then $\Lambda_+(p,e)$ can be expressed as the 
	projection of a spectrahedron, i.e., the image of a spectrahedron under a linear map. This is because 
	the cone of matrix sums of squares in $\cS[x_1,\ldots,x_n]^{d}$ is a linear image of 
	the positive semidefinite cone.  
\end{remark}
To discuss the relationship between our work and that of Kummer, Plaumann, and Vinzant~\cite{kummer2015hyperbolic}, 
we introduce a slight variation on Definition~\ref{def:soshyp}.
\begin{definition}
\label{def:wsishyp}
If $p\in \Hyp_{n,d}(e)$, we say that $p$ is \emph{weakly SOS-hyperbolic with respect to $e$} if 
$D_ep(x)D_up(x) - p(x)D^2_{ue}p(x) = [B_{p,e}(x)[u]]_{1,1}$
 is a sum of squares for all $u\in \Lambda_+(p,e)$. 
\end{definition}
\begin{question}
Clearly, if $p$ is SOS-hyperbolic with respect to $e$, then it is weakly SOS-hyperbolic with respect to $e$. 
Under what additional assumptions on $p$ (if any) are these two notions actually equivalent?
\end{question}

In Proposition~\ref{prop:pdet}, to follow, we show  that if a power of $p$ has a definite determinantal representation, 
then $p$ is SOS-hyperbolic. This is (at least formally) a slight strengthening of
a result of Kummer, Plaumann, and Vinzant~\cite[Corollary 4.3]{kummer2015hyperbolic}, which can be rephrased as saying 
that if a power of $p$ has a definite determinantal representation, then $p$ is \emph{weakly} SOS-hyperbolic. 
Proposition~\ref{prop:pdet} also generalizes 
result of Netzer, Plaumann, and Thom~\cite[Theorem 1.6]{netzer2013determinantal}, which can be rephrased as saying 
that if a power of $p$ has a definite determinantal representation, then $H_{p,e}(x)[e]$ is a matrix sum of squares.
The proof presented here generalizes the argument of Netzer, Plaumann, and Thom.
\begin{proposition}
\label{prop:pdet}
If $p\in \Hyp_{n,d}(e)$ and there is a positive integer $\ell$ such that $p^\ell$ has a definite determinantal representation, 
then $p\in \SOSHyp_{n,d}(e)$. 
\end{proposition}
\begin{proof}
We will show that $H_{p,e}(x)[u]$ is a matrix sum of squares whenever $u\in
\Lambda_+(p,e)$. It then follows from Proposition~\ref{prop:sos-equiv} that $B_{p,e}(x)[u]$ is a matrix sum of
squares whenever $u\in \Lambda_+(p,e)$.
Let $\tilde{p}(x) = p(x)^\ell$, and note that 
\[ \frac{D_u\tilde{p}(x+te)}{\tilde{p}(x+te)} = \ell\frac{D_u p(x+te) p(x+te)^{\ell-1}}{p(x+te)^{\ell}} = \ell \frac{D_u p(x+te)}{p(x+te)}.\]
In particular, $\ell H_m(p_x,D_up_x) = H_m(\tilde{p}_x,D_u\tilde{p}_x)$ for all $m\geq d\ell$. 
As such, it suffices to show that $H_{\tilde{p},e}(x)[u]$ is a matrix sum of squares.

Since $\tilde{p}(x)$ has a definite determinantal representation, we can write
\[ \tilde{p}(x) = \det(\mathcal{A}(x))\quad\textup{where}\quad \mathcal{A}(x) = \sum_{i=1}^{n}A_ix_i\]
for real symmetric matrices $A_1,\ldots,A_{n}$ such that  $\mathcal{A}(e)  = I$. 
If $u\in \Lambda_+(p,e) = \Lambda_+(\tilde{p},e)$, then  $\mathcal{A}(u) \psd 0$ and so has a positive semidefinite square root $\mathcal{A}(u)^{1/2}$.
In Example~\ref{eg:hermdet} we explicitly computed the matrix $H_{\tilde{p},e}(x)[u]$ and obtained 
\[ H_{\tilde{p},e}(x)[u]_{ij} = \tr\left[\mathcal{A}(x)^{i-1} \mathcal{A}(u)\mathcal{A}(x)^{j-1}\right] = 
\langle \mathcal{A}(u)^{1/2}\mathcal{A}(x)^{i-1},\mathcal{A}(u)^{1/2}\mathcal{A}(x)^{j-1}\rangle.\]
Here the inner product is $\langle X,Y\rangle = \tr(X^TY)$. This is clearly a Gram matrix with factors that are polynomials in $x$, and so $H_{\tilde{p},e}(x)[u]$ is a matrix
sum of squares whenever $u\in \Lambda_+(p,e)$. 
\end{proof}

All of the equality signs in Table~\ref{tab:nnh} follow directly from known results about when powers of hyperbolic polynomials
have definite determinantal representations. In what follows, for the cases $d=2$ and $n=3$ we give direct proofs, avoiding 
arguments about determinantal representations. It would be interesting to find a similar direct argument if $(n,d) = (4,3)$.  
\begin{proposition}
If $d=2$ or $n=3$ or $(n,d)=(4,3)$, then $\Hyp_{n,d}(e) = \SOSHyp_{n,d}(e)$.
\end{proposition}
\begin{proof}
The case $(n,d) = (4,3)$ follows from a result of 
Buckley and Ko\v{s}ir~\cite[Theorem 6.4]{buckley2007determinantal}, 
which says that the square of any smooth hyperbolic cubic form in $4$ variables
has a definite determinantal representation. Combining this with the fact that smooth hyperbolic polynomials
are dense in all hyperbolic polynomials, it follows that the square of any hyperbolic cubic in four variables
has a definite determinantal representation by a  
limiting argument~\cite[Proof of Corollary 4.10]{plaumann2013determinantal}. 

The case $n=3$ follows from the celebrated Helton-Vinnikov
theorem~\cite{helton2007linear} that (in its homogeneous
form~\cite{lewis2005lax}) says that $p$ has a definite determinantal
representation. Here we give an alternative direct argument.  If we choose a
basis $e,e',e''$ for $\RR^3$, then, by Proposition~\ref{prop:eshift}, \[
B_{p,e}(x_0e+x_1e'+x_2e'')[u] = K(x_0)B_{p,e}(x_1e'+x_2e'')[u]K(x_0)^T\] for
some polynomial matrix $K(x_0)$. As such, it suffices to show that if $u\in
\Lambda_+(p,e)$, then  $B_{p,e}(x_1e'+x_2e'')[u]$ is a matrix sum of squares.
This is a positive semidefinite matrix-valued polynomial for which each entry
is a homogeneous form in $x_1,x_2$.  From~\cite[Remark
5.10]{blekherman2016sums} it is known that all such polynomial matrices are
matrix sums of squares.   

In the case $d=2$, we again give a direct argument. First note that 
if $u\in \Lambda_+(p,e)$, then $H_{p,x}(x)[u] \psd 0$ for all $x$. Moreover, since $d=2$ there is a positive constant $c_0$, 
and polynomials $c_1(x)$ and $c_2(x)$ homogeneous of degree one and two respectively, such that 
\[ H_{p,e}(x)[u] = \begin{bmatrix} c_0 & c_1(x)\\c_1(x) & c_2(x)\end{bmatrix} = 
\begin{bmatrix} 1 & 0\\c_1(x)/c_0 & 1\end{bmatrix} \begin{bmatrix} c_0 & 0\\0 & c_2(x) - c_1(x)^2/c_0\end{bmatrix}
\begin{bmatrix} 1 & c_1(x)/c_0\\0 & 1\end{bmatrix}.\]
As such, $H_{p,e}(x)[u]$ is a matrix sum of squares if, and only if, the nonnegative quadratic form $c_2(x) - c_1(x)^2/c_0$ 
is a sum of squares. Since any nonnegative quadratic form is a sum of squares, we are done.
\end{proof}

\subsection{Hyperbolic certificates that are not sums of squares: $(n,d)=(4,4)$}

In this section we give an explicit example of a polynomial $p$ of degree four in four variables that is 
hyperbolic, but not SOS-hyperbolic, with respect to $e$. The example is the specialized V\'amos polynomial
\begin{equation}
	\label{eq:svamos} 
	p_{G_4}(x_1,x_2,x_3,x_4) = x_3^2x_4^2 + 4(x_1x_2x_3+x_1x_2x_4+x_1x_3x_4+x_2x_3x_4)(x_1+x_2+x_3+x_4),
\end{equation}
which is hyperbolic with respect to $(1,1,1,1)$ and has hyperbolicity cone that contains the nonnegative orthant. 
This is one of a much larger class of hyperbolic polynomials constructed from $k$-uniform hypergraphs
by Amini and Br\"and\'en~\cite[Theorem 9.4]{amini2018non}. 
The name arises because the basis generating polynomial of the V\'amos matroid is 
\[ p_{V_8}(z_1,z_2,\ldots,z_8) = \sum_{B\in \mathcal{B}}\prod_{i\in B}z_i\]
where $\mathcal{B} = \binom{[8]}{4}\setminus\{\{1, 2,
3, 4\}, \{1, 2, 5, 6\}, \{1, 2, 7, 8\}, \{3, 4, 5, 6\}, \{5, 6, 7, 8\}\}$. The specialized V\'amos polynomial is
the restriction $p_{G_4}(x_1,x_2,x_3,x_4) = p_{V_8}(x_1,x_1,x_2,x_2,x_3,x_3,x_4,x_4)$ 
of $p_{V_8}$ to a four-dimensional subspace. It is known that no power of 
$p_{V_8}$ has a definite determinantal representation~\cite{branden2011obstructions} 
and that the same holds for $p_{G_4}$~\cite{kummer2016note}. 

Kummer, Plaumann, and Vinzant~\cite[Example 5.11]{kummer2015hyperbolic} showed that the hyperbolic 
polynomial of degree four in five variables obtained by restricting $p_{V_8}$ to the subspace $z_1=z_2,z_3=z_4,z_5=z_6$
is not SOS-hyperbolic (in our language). The following result shows that the polynomial is still not SOS-hyperbolic
when further restricted to $p_{G_4}$.
\begin{proposition}
\label{prop:svamos}
	Let $p_{G_4}$ be the polynomial defined in~\eqref{eq:svamos}. Then $p_{G_4}$ is hyperbolic with respect to 
	$e\in (0,0,1,1)$, and $u = (0,0,0,1) \in \Lambda_+(p_{G_4},e)$, and yet 
	$B_{p_{G_4},e}(x)[u]$ is not a matrix sum of squares.
\end{proposition}
\begin{proof}
From~\cite[Theorem 9.4]{amini2018non} we know that $p_{G_4}$ is 
hyperbolic with respect to any point in the nonnegative orthant, so $u\in
\Lambda_+(p_{G_4},e)$. Since $p_{G_4}(e) > 0$ it follows that $e$ is a
direction of hyperbolicity for $p_{G_4}$. 
We will show that $[B_{p_{G_4},e}(x)[u]]_{11}$ is not a sum of squares when restricted
to the subspace $x_1+x_2+x_3+x_4=0$. Explicitly, consider the ternary sextic form 
\begin{align*}
	 q(x_1,x_2,x_3) & = e_1^TB_{p_{G_4},e}(x_1,x_2,x_3,-x_1-x_2-x_3)[u]e_1\\
	& = 32x_1^4x_2^2 + 56x_1^4x_2x_3 + 28x_1^4x_3^2 + 64x_1^3x_2^3 + 168x_1^3x_2^2x_3 +
168x_1^3x_2x_3^2 +\\
& \quad 64x_1^3x_3^3 + 32x_1^2x_2^4 + 168x_1^2x_2^3x_3 + 280x_1^2x_2^2x_3^2 + 176x_1^2x_2x_3^3 + 46x_1^2x_3^4 +\\
& \quad 56x_1x_2^4x_3 + 168x_1x_2^3x_3^2 + 176x_1x_2^2x_3^3 + 76x_1x_2x_3^4 + 12x_1x_3^5 + 28x_2^4x_3^2 + \\
&\quad 64x_2^3x_3^3 + 46x_2^2x_3^4 + 12x_2x_3^5 + 2x_3^6.
\end{align*}
We will show that $q$ is not a sum of squares by constructing a separating hyperplane with rational coefficients. 
Ordering the monomials of $q$ as above, we write $q = \sum_{i=1}^{22}c_ix^{\alpha_i}$
where 
\[
 c = (32,56,28,64,168,168,64,32,168,280,176,46,56,168,176,76,12,28,64,46,12,2).
\]
If $\mathcal{N}$ is the Newton polytope of $q$ (the convex hull of the $22$ 
exponent vectors $\alpha_i$), then the extreme points of $\mathcal{N}$ are the integer points
$(4,2,0), (2,4,0), (4,0,2), (0,4,2), (0,0,6)$. 
As such, if $q = \sum_i q_i^2$ is a sum of squares, then the 
$q_i$ are in the subspace spanned by the monomials with exponent vectors in $\frac{1}{2}\mathcal{N}$, i.e., 
 $\{x_1^2x_2,x_1x_2^2,x_1^2x_3,x_1x_3^2,x_1x_2x_3,x_1x_3^2,x_2x_3^2,x_3^3\}$.
Furthermore, since $q(1,-1,0) = q(1,0,-1) = q(0,1,-1)=0$ it follows that each $q_i$ must also vanish at these three points. 
As such, if $q = \sum_i q_i^2$, then the $q_i$ must be in the $5$-dimensional subspace spanned by the entries
\[ 
 m(x) = \begin{bmatrix} x_1x_2x_3 & (x_1+x_2)x_1x_2 & (x_1+x_3)x_1x_3 & (x_2+x_3)x_2x_3 & (x_1+x_2+x_3)x_3^2\end{bmatrix}^T
\]
and so, if $q$ were a sum of squares, then there would be $G \psd 0$ such that
\begin{equation}
	\label{eq:qsosid} \sum_{i=1}^{22}c_ix^{\alpha_i} = \tr(G m(x)m(x)^T).
\end{equation}
To show that this is impossible, define a linear functional $\ell$ on the span of the $x^{\alpha_i}$ by
\begin{multline*} (\ell(x^{\alpha_i}))_{i=1,\ldots,22} =  
 (81,-249,323,40,24,-186,32,81,24,233,-89,15,-249,\\
-186,-89,322,-412,323,32,15,-412,1186).\end{multline*}
The linear functional $\ell$ was obtained by solving a semidefinite feasibility problem numerically and rounding the solution.
It satisfies $\sum_{i=1}^{22}c_i\ell(x^{\alpha_i}) = -144$. 
If we apply $\ell$ to the entries of $m(x)m(x)^T$, we obtain
\[\ell( m(x)m(x)^T) = \begin{bmatrix} 233 & 48 & -275 & -275 & 144\\48 & 242 & -178 & -178 & -84\\
-275 & -178 & 402 & 377 & -117\\
-275 & -178 & 377 & 402 & -117\\
144 & -84 & -117 & -117 & 212\end{bmatrix} \psd 0.\]
Applying $\ell$ to both sides of~\eqref{eq:qsosid} would give,
$-144 = \sum_{i=1}^{22}c_i\ell(x^{\alpha_i}) = \tr(G \ell(m(x)m(x)^T)) \geq 0$,
a contradiction.
\end{proof}
\subsection{Increasing the degree and number of variables}

We now describe a particular way to take any hyperbolic polynomial $p$ that is not SOS-hyperbolic
and construct from it a hyperbolic polynomial that has larger degree and/or more variables, that is not SOS-hyperbolic.
We begin by showing how to increase the degree alone.
\begin{proposition}
	\label{prop:dv-increase}
	Suppose that $p\in \Hyp_{n,d}(e)$
	$u\in \Lambda_+(p,e)$ is such that $B_{p,e}(x)[u]$ is not a matrix sum of squares, and $\ell(x)$ is a linear form
	such that $\ell(e)>0$ and $\ell(u) = 0$. 
	Let $\tilde{p}(x) = \ell(x)^kp(x)$. Then $\tilde{p}\in \Hyp_{n,d+k}(e)$, $u\in \Lambda_+(\tilde{p},e)$,
	and $B_{\tilde{p},e}(x)[u]$ is not a matrix sum of squares.
\end{proposition}	
\begin{proof}
	Since $\ell(x)$ is linear and $\ell(e) > 0$ it follows that $\ell$ is hyperbolic with respect to $e$
	and $u\in \Lambda_+(\ell,e)$. Now $\Lambda_+(\tilde{p},e) = \Lambda_+(\ell,e) \cap \Lambda_+(p,e)$ so 
	$u\in \Lambda_+(\tilde{p},e)$. Furthermore, since $D_u\ell(x+te) = 0$, 
	\[ \tilde{p}_x(t) =\ell(x+te)^kp(x+te)\quad\textup{and}\quad D_u\tilde{p}_x = \ell(x+te)^kD_up(x+te).\]
	Then $\frac{D_u\tilde{p}_x}{\tilde{p_x}} = \frac{D_up_x}{p_x}$. As such, $H_{\tilde{p},e}(x)[u]$
	contains $H_{p,e}(x)[u]$ as the upper-left $d\times d$ submatrix. It follows that if $B_{p,e}(x)[u]$ is not
	a matrix sum of squares, then $H_{p,e}(x)[u]$, and consequently $H_{\tilde{p},e}(x)[u]$
	and $B_{\tilde{p},e}(x)[u]$, are not matrix sums of squares.	
\end{proof}
Next we discuss one way to take a hyperbolic polynomial that is not SOS-hyperbolic and construct a new hyperbolic polynomial of the 
same degree in more variables that is not SOS-hyperbolic. We begin with a simple preliminary result.
\begin{proposition}
	\label{prop:n-increase}
	If $p\in \SOSHyp_{n,d}(e)$ and $L$ is a subspace of $\RR^n$ such that $e\in L$, then 
	$\left. p\right|_{L}$ (the restriction of $p$ to $L$) is SOS-hyperbolic with respect to $e$.
\end{proposition}
\begin{proof}
If $u\in L$, then it is straightforward to check that $B_{\left.p\right|_L,e}(x)[u]$ is the restriction of $B_{p,e}(x)[u]$ to $L$. As such, if $B_{p,e}(x)[u]$ is a 
matrix sum of squares for all $u\in \Lambda_+(p,e)$, it follows that $B_{\left.p\right|_{L},e}(x)[u]$ is a matrix sum of squares
for all $u\in \Lambda_+(\left.p\right|_{L},e) = \Lambda_+(p,e)\cap L$. 
\end{proof}
The following simple construction of hyperbolic polynomials is 
a special case of, e.g., the \emph{additive convolution} (or \emph{finite free convolution}) of hyperbolic polynomials~\cite{marcus2015finite}. It involves two disjoint sets of $n$ and $n'$ variables, respectively denoted by $x$ and $x'$. 
\begin{lemma}
If $p\in \Hyp_{n,d}(e)$ and $q\in \Hyp_{n',1}(e')$,  then 
\[ \tilde{p}(x,x') = q(e')p(x) + q(x')D_ep(x) \in \Hyp_{n+n',d}((e,e'))\]
and $\Lambda_{+}(\tilde{p},(e,e')) \supseteq \{(x,x'): x\in \Lambda_{+}(p,e),\; x'\in\Lambda_+(q,e')\}$.
\end{lemma}
\begin{proof}
If $w(x,x') = p(x)q(x')$ then $w$ is hyperbolic with respect to $(e,e')$ and its hyperbolicy cone 
is the Cartesian product of $\Lambda_+(p,e)$ and $\Lambda_+(q,e')$. 
Furthermore
\[ \tilde{p}(x,x') = D_{(e,e')}w(x,x')\]
and so $\Lambda_{+}(\tilde{p},(e,e')) \supseteq \Lambda_+(w,(e,e'))$~\cite{renegar2006hyperbolic}.
\end{proof}
\begin{proposition}
\label{prop:nup}
If $p\in \Hyp_{n,d}(e)$ is not SOS-hyperbolic with respect to $e$ and $q\in \Hyp_{n',1}(e')$, then 
$\tilde{p}(x,x') = q(e')p(x) + q(x')D_ep(x)\in \Hyp_{n+n',d}((e,0))$ is not SOS-hyperbolic with respect to $(e,0)$.
\end{proposition}
\begin{proof}
First we note that $(e,0)\in \Lambda_{++}(\tilde{p},(e,e'))$ because $\tilde{p}(e,0) = q(e')p(e) > 0$ 
and $(e,0)\in \Lambda_+(\tilde{p},(e,e'))$. 
If $L\subseteq \RR^{n+n'}$ is the subspace spanned by the first $n$ coordinate directions, then  
$\left.\tilde{p}\right|_{L} = p$ and $(e,0)\in L$. Hence, if $p\notin \SOSHyp_{n,d}(e)$, it follows that 
$\tilde{p}\notin \SOSHyp_{n+n',d}((e,0))$.
\end{proof}

We now construct, when $n\geq 4$ and $d\geq 4$, an explicit hyperbolic polynomial
and direction of hyperbolicity that is not SOS-hyperbolic with respect to that direction.
\begin{theorem}
If $n\geq 4$ and $d\geq 4$ and $\tilde{e} = (0,0,1,1,0,\ldots,0)$, there exists a polynomial 
$p\in \Hyp_{n,d}(\tilde{e})$ that is not SOS-hyperbolic with respect to $\tilde{e}$.
\end{theorem}
\begin{proof}
From Proposition~\ref{prop:svamos}, we know that if $e = (0,0,1,1)$ and $u = (0,0,0,1)$, then 
the specialized V\'amos polynomial $p_{G_4}$ is hyperbolic, but not SOS-hyperbolic, with respect to $e$.
Let $\tilde{e} = (0,0,1,1,0,\ldots,0)$ and $\tilde{u} = (0,0,0,1,0,\ldots,0)$. 

We can then use Propositions~\ref{prop:nup} and~\ref{prop:dv-increase} to construct $p$. 
Explicitly, we let $q(x') = x_5+\cdots + x_n$ and $e' = (1,1,\ldots,1)/(n-4)$ in Proposition~\ref{prop:nup},
and let $\ell(x) = x_3$ in Proposition~\ref{prop:dv-increase} (which satisfies $\ell(\tilde{e}) = 1$ and $\ell(\tilde{u}) = 0$)
to obtain
\[ p(x_1,\ldots,x_n) = x_3^{d-4}(p_{G_4}(x_1,x_2,x_3,x_4) + (x_5+\cdots + x_n)D_{e}p_{G_4}(x_1,x_2,x_3,x_4)).\]
This is hyperbolic with respect to $\tilde{e}$ but is not SOS-hyperbolic with respect to $\tilde{e}$ because $B_{p,\tilde{e}}(x)[\tilde{u}]$ is not a matrix sum of squares.
\end{proof}
We conclude the section with a natural question raised by one of the referees. We know that if $p$ is hyperbolic with respect to $e$, 
then it is hyperbolic with respect to any $e'$ in the associated open hyperbolicity cone.
\begin{question}
Is there a polynomial $p$, together with directions $e$ and $e'$, 
such that $p$ is SOS-hyperbolic with respect to $e$ and $e'$ is in the associated open hyperbolicity cone,   
and yet $p$ is not SOS-hyperbolic with respect to $e'$?
\end{question}
If such a polynomial were to exist, it would also be an example of an SOS-hyperbolic polynomial for which no 
power has a definite determinantal representation.
 
\section{Cubic hyperbolic polynomials}
\label{sec:cubics}
In this section, we focus on hyperbolic cubics. The simple structure of the
discriminant of a univariate cubic means that there is an explicit connection
between hyperbolicity of cubic forms and the maximum value of general cubic
forms on the unit sphere. Using this connection, in Section~\ref{sec:hard} we
show that, given a cubic form $p$ and a direction $e$, both with rational
coefficients,  it is co-NP hard to decide whether $p$ is hyperbolic with
respect to $e$. In Section~\ref{sec:cubic-sos} we construct an explicit cubic
form in $43$ variables and a direction of hyperbolicity $e$, so that $p$ is not
SOS-hyperbolic with respect to $e$. This gives a
hyperbolic cubic for which no power has a definite determinantal
representation. 

\subsection{Hyperbolicity and extreme values of cubics on the sphere}
In this section we focus on cubic polynomials of the form
\begin{equation}
\label{eq:std-cubic} p(x_0,x) = x_0^3 - 3x_0(x_1^2+\cdots + x_n^2) + 2q(x)
\end{equation}
for some $q\in \RR[x_1,\ldots,x_n]_3$, and fix the candidate direction of hyperbolicity to be $e_0$. 

The next result gives a characterization of when a cubic polynomial in the
form~\eqref{eq:std-cubic} is hyperbolic with respect to $e_0$, and a necessary
condition for it to be SOS-hyperbolic with respect to $e_0$.
\begin{proposition}
\label{prop:hyp3test}
If $p(x_0,x) = x_0^3 - 3x_0\|x\|^2 + 2q(x)$, where $q\in \RR[x_1,\ldots,x_{n}]_{3}$, 
then
\begin{align}
\label{eq:std-max} 
p\in \Hyp_{n+1,3}(e_0) \;\;& \iff\;\;\max_{\|x\|^2 =1}|q(x)| \leq 1\\
\;\;& \iff\;\; \|x\|^4 - 2zq(x) + z^2\|x\|^2 \geq 0\;\;\textup{for all $(x,z)\in \RR^{n}\times \RR$}.
\end{align}
Moreover, if $p\in \SOSHyp_{n+1,3}(e_0)$, then $\|x\|^4 - 2zq(x) + z^2\|x\|^2$ is a sum of squares.
\end{proposition}
\begin{proof}
Let $W = \{x\in \RR^{n+1}: x_0 = 0\}$ and note that $e_0\notin W$. By Corollary~\ref{cor:npt2}, 
 $p$ is hyperbolic with respect to $e_0$ if, and only if, $B_{p,e_0}(0,x)[e_0] \psd 0$ for all $x\in \RR^{n}$. 
An explicit computation shows that 
\[ B_{p,e_0}(0,x)[e_0] = \begin{bmatrix} 9\|x\|^4 & -6q(x) & -3\|x\|^2\\-6q(x) & 6\|x\|^2 & 0\\-3\|x\|^2 & 0 & 3\end{bmatrix}.\]
By taking a Schur complement and dividing by $6$, 
\begin{equation}
	\label{eq:mat-B} B_{p,e_0}(0,x)[e_0] \psd 0\;\;\iff\;\;\begin{bmatrix} \|x\|^4 & -q(x)\\-q(x) & \|x\|^2\end{bmatrix} \psd 0 \;\;\iff\;\;\max_{\|x\|^2=1}|q(x)|\leq 1.
\end{equation}
Similarly~\eqref{eq:mat-B} holds for all $x\in \RR^n$ if, and only if, 
\[ \begin{matrix}\begin{bmatrix} 1 & z\end{bmatrix}\\\phantom{\begin{bmatrix} 1 & z\end{bmatrix}}\end{matrix} \begin{bmatrix} \|x\|^4 & -q(x)\\-q(x) & \|x\|^2\end{bmatrix}\begin{bmatrix}1\\z\end{bmatrix}
= \|x\|^4 - 2zq(x) + z^2\|x\|^2\geq 0\]
for all $(x,z)\in \RR^{n}\times \RR$.
If $p\in \SOSHyp_{n+1,3}(e_0)$, then $B_{p,e_0}(0,x)[e_0]$ is a matrix sum of squares, and 
\[ \begin{matrix} \begin{bmatrix}1& z & \|x\|^2\end{bmatrix}\\
	\phantom{\begin{bmatrix}1& z & \|x\|^2\end{bmatrix}}\\
	\phantom{\begin{bmatrix}1& z & \|x\|^2\end{bmatrix}}\end{matrix}\begin{bmatrix} 9\|x\|^4 & -6q(x) & -3\|x\|^2\\-6q(x) & 6\|x\|^2 & 0\\-3\|x\|^2 & 0 & 3\end{bmatrix}\begin{bmatrix} 1\\z\\\|x\|^2\end{bmatrix} = 6(\|x\|^4  -2zq(x)+z^2\|x\|^2)\]
is a sum of squares.
\end{proof}
This connection between (SOS-)hyperbolicity and (sum of squares relaxations of) the extreme values of cubic forms on the unit sphere, 
is central to the remaining discussion in this section.

\subsection{Hardness of testing hyperbolicity of cubics}
\label{sec:hard}
We have seen that for cubics in the form~\eqref{eq:std-cubic}, to test hyperbolicity we need to be able to 
compute the extreme values of a cubic form on the unit sphere. This connection will allow us to establish
NP-hardness of deciding hyperbolicity of cubic forms.
We begin with a result of Nesterov~\cite[Theorem 4]{nesterov2003random} stating that we can 
compute the size of the largest clique in a graph by
maximizing an associated cubic form over the unit sphere.
\begin{theorem}[Nesterov]
\label{thm:nesterov}
Let $\cG = (\cV,\cE)$ be a simple graph (i.e., without loops and multiple edges) 
and let $\omega(\cG)$ be the size of a largest clique in $\cG$. Define
\begin{equation}
\label{eq:qG}
	 q_{\cG}(x,y) = \sum_{(i,j)\in \cE} x_ix_jy_{ij}.
\end{equation}
Then 
\[ \max_{\|x\|^2+\|y\|^2=1} q_\cG(x,y) = \sqrt{\frac{2}{27}} \sqrt{1-\frac{1}{\omega(\cG)}}.\]
Moreover, if $\cC\subset \cV$ is a clique with $|\cC|=\omega(\cG)$, then $q_{\cG}$ is maximized whenever 
\[ x_{i}^2 = \begin{cases} \frac{2}{3\omega(\cG)} & \textup{if $i\in \cC$}\\0 & \textup{otherwise}\end{cases}
\quad\textup{and}\quad y_{ij}^2 = \begin{cases}\frac{1}{3\binom{\omega(\cG)}{2}} & \textup{if $i,j\in \cC$}\\0 & \textup{otherwise}
	\end{cases}\]
and $\textup{sign}(y_{ij}) = \textup{sign}(x_i)\textup{sign}(x_j)$.
\end{theorem}
Nesterov obtained this by reformulating Motzkin and Straus' celebrated formula for the maximum size of a clique in a graph
in terms of the maximum value of a certain quadratic form over the unit simplex~\cite{motzkin1965maxima}. 
Combining Nesterov's result and Proposition~\ref{prop:hyp3test} allows us to construct a family of hyperbolic cubic
forms from any simple graph.
\begin{proposition}
	Given a simple graph $\cG = (\cV,\cE)$ and an integer $k\geq 2$, define a cubic form in $|\cV|+|\cE|+1$ variables by 
\[ p_{\cG,k}(x_0,x,y) = \frac{2k}{k-1}x_0^3 - x_0(\|x\|^2+\|y\|^2) + q_{\cG}(x,y).\]
Then $p_{\cG,k}$ is hyperbolic with respect to $e_0$ if, and only if, $\omega(\cG) \leq k$.
\end{proposition}
\begin{proof}
We make the change of variables $\tilde{x}_0 = \sqrt{\frac{6k}{k-1}}x_0$ to obtain
\[ \tilde{p}_{\cG,k}(\tilde{x}_0,x,y) =  \frac{1}{\sqrt{27}}\sqrt{\frac{k-1}{2k}}\left[\tilde{x}_0^3 - 3\tilde{x}_0(\|x\|^2 + \|y\|^2) +2\left(\frac{q_{\cG}(x,y)}{\sqrt{\frac{2}{27}}\sqrt{1-\frac{1}{k}}}\right)\right]\]
which is hyperbolic with respect to $e_0$ if, and only if, $p_{\cG,k}$ is hyperbolic with respect to $e_0$. 
The result then follows immediately from Proposition~\ref{prop:hyp3test} and Theorem~\ref{thm:nesterov}.
\end{proof}
Given a positive integer $k$ and a simple graph $\cG = (\cV,\cE)$, the problem of deciding whether $\omega(\cG) \geq k+1$ is 
NP-hard~\cite{karp1972reducibility}. This immediately gives the following hardness result.
\begin{theorem}
\label{thm:np}
Given a homogeneous cubic polynomial $p$ and a direction $e$, both with rational coefficients, 
the decision problem ``Is $p$ hyperbolic with respect to $e$?'' is co-NP hard.
\end{theorem}
This result is strong evidence that hyperbolic cubics that are not SOS-hyperbolic should exist. Indeed, 
if all hyperbolic cubics were SOS-hyperbolic, we could check hyperbolicity of cubics (and so maximize cubic forms over 
the unit sphere) by solving semidefinite programs of size polynomial in the input. 

\subsection{A cubic that is hyperbolic but not SOS-hyperbolic}
\label{sec:cubic-sos}
In this section we construct an explicit hyperbolic cubic that is not SOS-hyperbolic. 
In light of Proposition~\ref{prop:hyp3test}, the basic approach will be to construct a 
cubic polynomial $q$ for which $\|x\|^4 - 2zq(x) +z^2\|x\|^2$ is nonnegative 
but not a sum of squares. Drawing on the results of the previous section, 
we consider graphs $\cG$ with maximum clique of size $\omega(\cG)$ 
and polynomials of the form
\begin{equation}
\label{eq:pG} p_{\cG,\omega(\cG)}(x_0,x,y) = x_0^3 - 3x_0(\|x\|^2+\|y\|^2) + 2\left(\frac{q_{\cG}(x,y)}{\sqrt{\frac{2}{27}}\sqrt{1-\frac{1}{\omega(\cG)}}}\right).
\end{equation}
By Proposition~\ref{prop:hyp3test}, if $p_{\cG,\omega(\cG)}$ is SOS-hyperbolic with respect to $e_0$, then 
\[ r_{\cG}(x,y,z) := (\|x\|^2+\|y\|^2)^2 - 2\left(\frac{q_{\cG}(x,y)}{\sqrt{\frac{2}{27}}\sqrt{1-\frac{1}{\omega(\cG)}}}\right) + z^2(\|x\|^2+\|y\|^2)\]
is a sum of squares. 

The following result tells us that, if $r_{\cG}$ is a sum of squares, then there
exists a $(|\cV|+|\cE|)\times(|\cV|+|\cE|)$ correlation matrix (positive
semidefinite matrix with unit diagonal) with nullspace containing a set of
vectors $(x_C,y_C)$ corresponding to each of the maximum cliques of the graph.
(A scaled version of one such vector, for a particular graph, is illustrated in Figure~\ref{fig:basesa}.)
\begin{proposition}
\label{prop:elliptope}
Let $\cG=(\cV,\cE)$ be a simple graph with maximum clique size $\omega(\cG)$. 
For each maximum clique $C$ of $\cG$, let 
\[ \left(x_C,y_C\right)
= \left(\sum_{v\in \cV(C)} \frac{2}{3\omega(\cG)}e_v, \sum_{\{v,w\}\in \cE(C)}\frac{1}{3\binom{\omega(\cG)}{2}}e_{\{v,w\}}\right)\in \RR^{\cV}\times \RR^{\cE}.\]
If $r_{\cG}$ is a sum of squares, then there exists a $(|\cV|+|\cE|)\times (|\cV|+|\cE|)$ positive semidefinite matrix $X$
such that $X_{ii} = 1$ for all $i=1,2,\ldots,|\cV|+|\cE|$ and $\left(x_C, y_C\right)$ is in the nullspace of $X$
for all maximum cliques $C$ of $\cG$. 
\end{proposition}
\begin{proof}
Let $m_0(x,y) = \begin{bmatrix} (x_i^2)_{i\in \cV}& (y^2_{ij})_{\{i,j\}\in \cE}\end{bmatrix}^T$ and
\[ m_1(x,y,z) = \begin{bmatrix}(zx_i)_{i\in \cV}&(zy_{ij})_{\{i,j\}\in \cE}&(x_ix_j)_{i,j\in \cV}&(y_{ij}y_{k\ell})_{\{i,j\},\{k,\ell\}\in \cE}&(x_iy_{jk})_{i\in \cV,\{j,k\}\in \cE}\end{bmatrix}^T.\]
If $r_{\cG}$ were a sum of squares, then there would exist  
$G = \left[\begin{smallmatrix}G_{00} & G_{01}\\G_{01}^T & G_{11}\end{smallmatrix}\right] \psd 0$ such that
\[ r_{\cG}(x,y,z) = m_0(x,y)^TG_{00}m_0(x,y) + 2m_{0}(x,y)^TG_{01}m_1(x,y,z) + m_1(x,y,z)^TG_{11}m_1(x,y,z).\]
This holds because there is no term of the form $z^4$ in $r_{\cG}$ and so the monomial $z^2$ has zero coefficient in 
any polynomial appearing in a sum of squares
decomposition. By comparing coefficients of $x_i^4$ and $y_{ij}^4$, we see that $[G_{00}]_{ii} = 1$ for $i=1,2,\ldots,|\cV|+|\cE|$.

The group $\ZZ_{2}^{|\cV|+1}$ acts on $\RR^{|\cV|+|\cE|+1}$ by 
\[ (\epsilon_0,\epsilon_1,\ldots,\epsilon_{|\cV|}) \cdot ((x_i)_{i\in \cV},(y_{ij})_{\{i,j\}\in \cE},z) = 
((\epsilon_i x_i)_{i\in \cV},(\epsilon_0\epsilon_i\epsilon_j y_{ij})_{\{i,j\}\in \cE}, \epsilon_0 z)\]
where $\epsilon_i\in \{-1,1\}$ for $i=0,1,\ldots,|\cV|$. 
This induces an action on polynomials which leaves $r_{\cG}(x,y,z)$ invariant.
Averaging over this action, we can see that if $r_{\cG}$ is a sum of squares, then it must have a Gram matrix with $G_{01} = 0$. 
This is because every monomial in $m_0$ is fixed by the action, and no monomial in $m_1$ is fixed by the action. 
As such, there must be positive semidefinite matrices
$G_{00}$ and $G_{11}$ such that 
\[ r_{\cG}(x,y,z) = m_0(x,y)^TG_{00}m_0(x,y) + m_1(x,y,z)^TG_{11}m_1(x,y,z)\]
and $[G_{00}]_{ii} = 1$ for all $i=1,2,\ldots,|\cV|+|\cE|$. 
From Nesterov's result (Theorem~\ref{thm:nesterov}), one can check that
if $C$ is a maximum clique, then $r_{\cG}(x,y,z)=0$ whenever $z=1$ and $m_0(x,y)^T = \left[\begin{smallmatrix}x_C\\ y_C\end{smallmatrix}\right]$.
It follows that 
$\left[\begin{smallmatrix} x_C\\ y_C\end{smallmatrix}\right]^T G_{00}\left[\begin{smallmatrix} x_C\\y_C\end{smallmatrix}\right] = 0$ for all maximum cliques $C$. Since $G_{00}\psd 0$ it follows that, in fact, 
$G_{00}\left[\begin{smallmatrix} x_C \\ y_C\end{smallmatrix}\right] = 0$ for all maximum cliques $C$. Taking $X =G_{00}$ completes the proof. 
\end{proof}
We now focus on the case in which $\cG$ is the icosahedral graph. This graph has $12$ vertices, $30$ edges, and
$20$ triangles, all of which are maximum cliques. As such, the
polynomial 
\begin{equation}
	\label{eq:pG3}
	p_{\cG,3}(x_0,x,y) = x_0^3 - 3x_0(\|x\|^2+\|y\|^2) + 9q_{\cG}(x,y)
\end{equation}
has  $12+30+1=43$ variables and is hyperbolic with respect to $e_0$. In the
proof of Theorem~\ref{thm:icos3}, we call a pair of vertices of the icosahedral
graph \emph{antipodal} if they are at distance three in the graph. The twelve
vertices can be partitioned into six pairs of antipodal vertices.
\begin{theorem}
\label{thm:icos3}
If $\cG=(\cV,\cE)$ is the icosahedral graph, then $p_{\cG,3}$, defined in~\eqref{eq:pG3},
is hyperbolic, but not SOS-hyperbolic, with respect to $e_0$.
\end{theorem}
\begin{proof}
By Proposition~\ref{prop:elliptope}
it suffices to show that there is no $42\times 42$ correlation matrix with all of the 
vectors $(x_C,y_C)$, ranging over maximum cliques $C$, in its nullspace. 

\begin{figure}
\begin{center}
\subfloat[]{\label{fig:basesa}\includegraphics[scale=0.45]{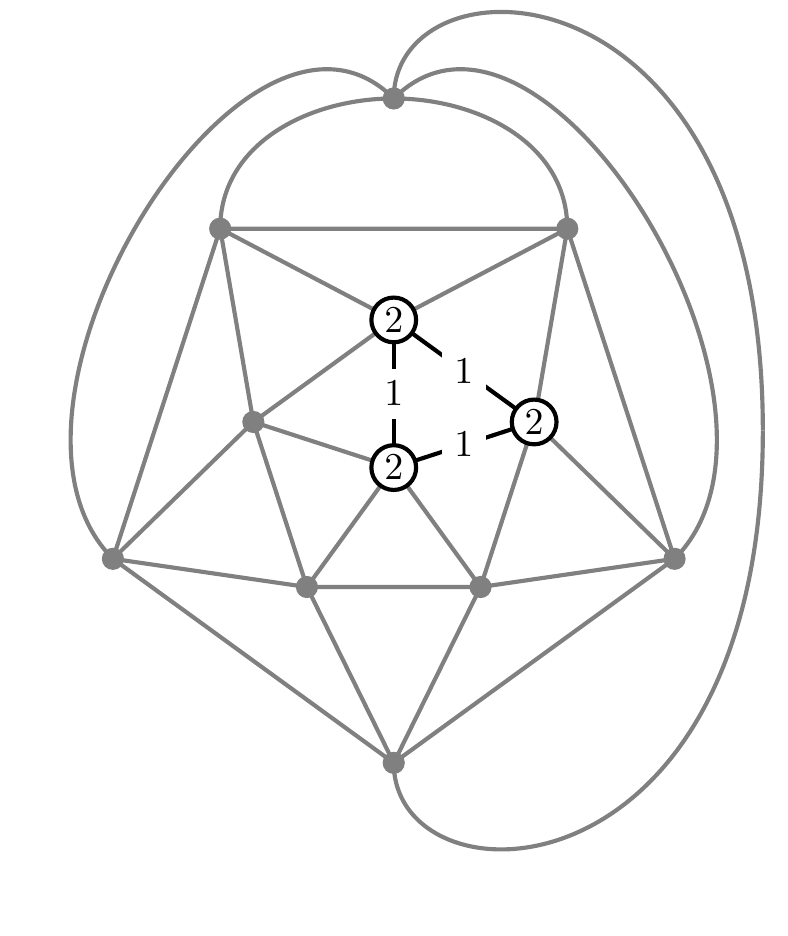}}
\subfloat[]{\label{fig:basesb}\includegraphics[scale=0.45]{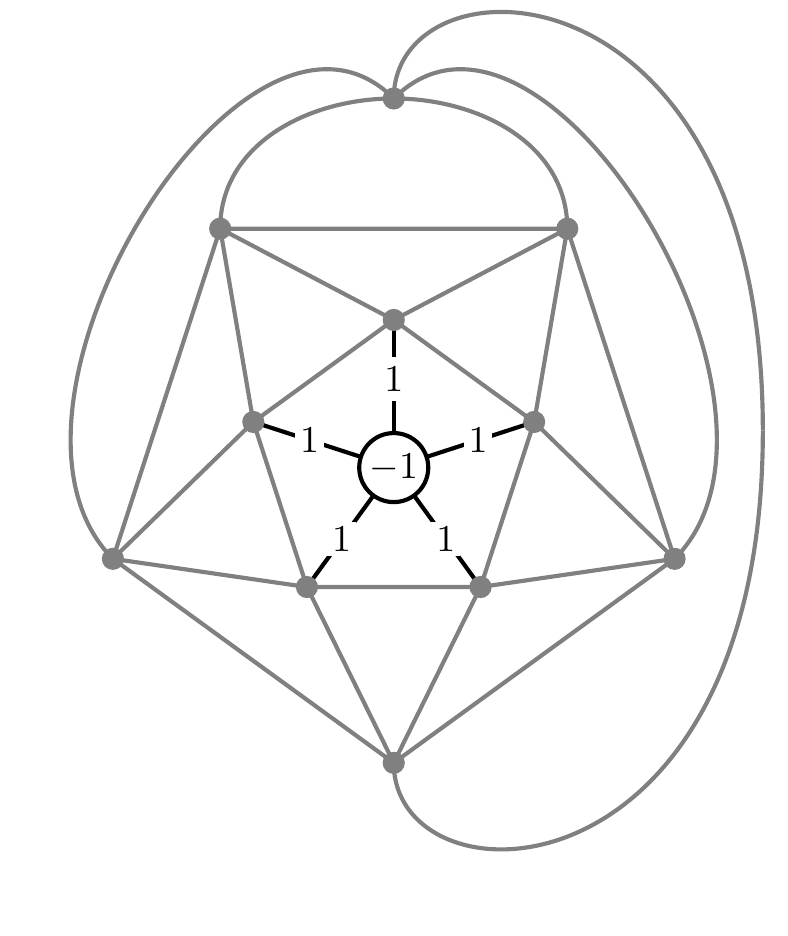}}
\subfloat[]{\label{fig:basesc}\includegraphics[scale=0.45]{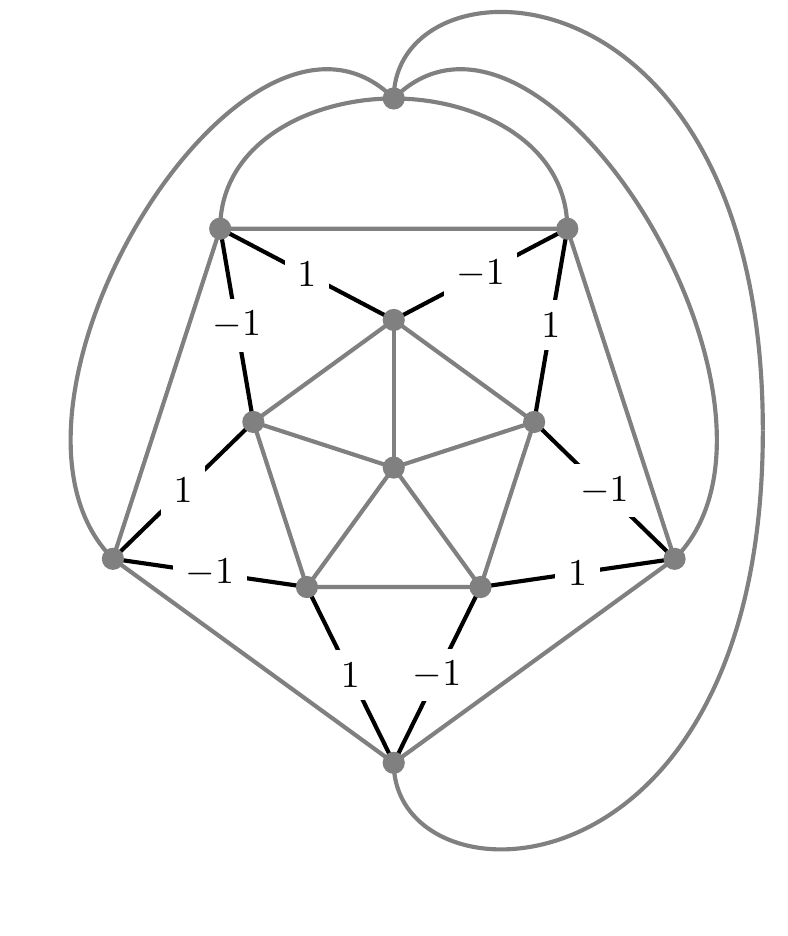}}
\subfloat[]{\label{fig:basesd}\includegraphics[scale=0.45]{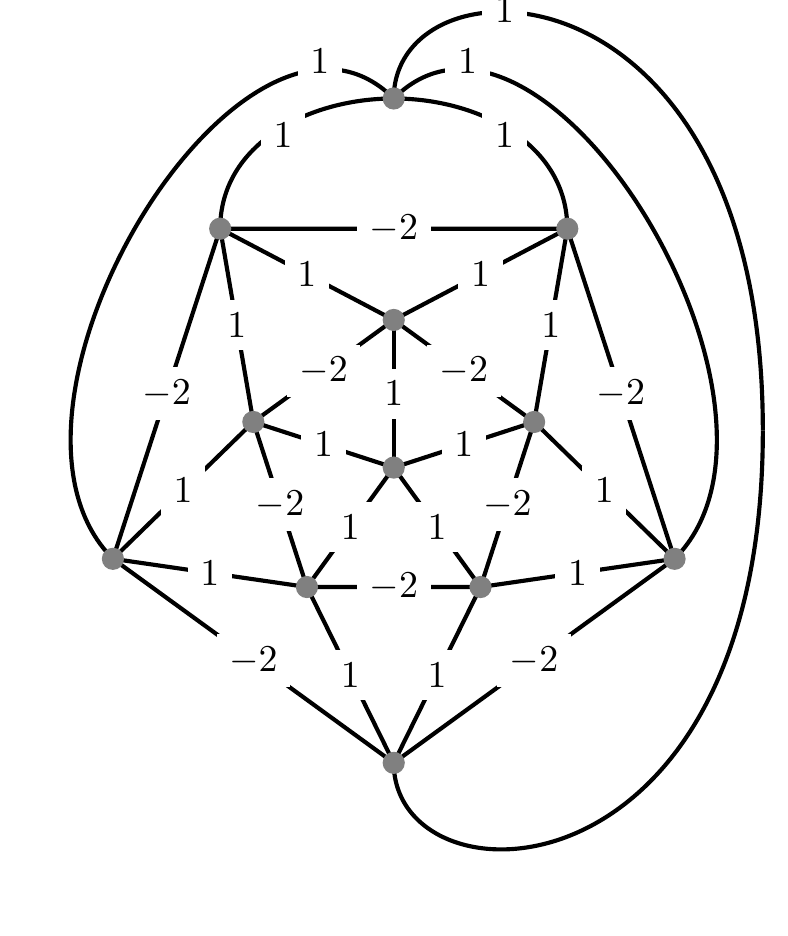}}
\end{center}
\caption{\label{fig:bases} The diagrams illustrate certain vectors in $\RR^{42} = \RR^{|\cV|}\times \RR^{|\cE|}$ thought of 
as vertex and edge labellings of the icosahedral graph.
The labels specify the nonzero coordinates. Unlabeled vertices 
or edges (in gray) correspond to zero coordinates.
Figure (a) is $(x_C,y_C)$ for a particular maximum clique in the icosahedral graph. 
Figures (b), (c), and (d) show certain vectors that are orthogonal to $(x_C,y_C)$ for all maximum cliques. }
\end{figure}

The $22$-dimensional orthogonal complement of the $20$-dimensional span of the
$(x_C,y_C)$ has a basis given by the twelve
vectors (one for each vertex) shown in Figure~\ref{fig:basesb}, any five of the
six vectors (one for each pair of antipodal vertices), shown in Figure~\ref{fig:basesc}, and any five of the six vectors (one for each pair of
antipodal vertices), shown in Figure~\ref{fig:basesd}.  Let $V$ be the $42
\times 22$ matrix with these vectors as columns. One can check (e.g., by
computing a rational LDL decomposition) that 
\[ V^T \begin{bmatrix} -11I_{12\times 12} & 0\\0 & 4I_{30\times 30}\end{bmatrix}V \psd 0. \]
If there were a correlation matrix with all of the $(x_C,y_C)$ in its
nullspace, then it could be written in the form $VMV^T$ where $M \psd 0$ and $[VMV^T]_{ii} = 1$ for all $i=1,2,\ldots,42$. 
But then 
\[ -12 =  \tr\left(\begin{bmatrix} -11I_{12\times 12} & 0\\0 & 4I_{30\times 30}\end{bmatrix} VMV^T\right) = 
	\tr\left((V^T\begin{bmatrix} -11I_{12\times 12} & 0\\0 & 4I_{30\times 30}\end{bmatrix}V)M\right),\]
and so $M$ cannot be positive semidefinite, and  no such correlation matrix can exist.
\end{proof}
The following is an immediate corollary of Theorem~\ref{thm:icos3} and Proposition~\ref{prop:pdet}. We state it explicitly 
because this appears to be the first known example of a cubic hyperbolic polynomial such that no power has 
a definite determinantal representation.
\begin{corollary}
If $\cG$ is the icosahedral graph, then no power of $p_{\cG,3}$, defined in~\eqref{eq:pG},
has a definite determinantal representation.
\end{corollary}
By using Proposition~\ref{prop:nup} together with our example of a hyperbolic, but not SOS-hyperbolic, cubic in $43$ variables (Theorem~\ref{thm:icos3}), we can construct such cubics whenever the number of variables is at least $43$. 
\begin{corollary}
\label{cor:d3}
If $n\geq 43$ and $e\in \RR^{n}$, then $\SOSHyp_{n,3}(e) \subsetneq \Hyp_{n,3}(e)$. 
\end{corollary}
Although we have only been able to construct a hyperbolic cubic in $43$ variables that is not SOS-hyperbolic,
it has been conjectured by Mario Kummer (personal communication), and seems likely, that such examples
should already exist with five variables.
\begin{conjecture}
\label{conj:mario}
If $n\geq 5$ and $e\in \RR^n$, then $\SOSHyp_{n,3}(e) \subsetneq\Hyp_{n,3}(e)$. 
\end{conjecture}
One way to resolve this conjecture, would be to find an explicit cubic form $q$ in four variables such that 
$\|x\|_2^4 - 2z q(x) + z^2\|x\|_2^2$ is nonnegative but is not a sum of squares. Then $p(x_0,x) = x_0^3-3x_0\|x\|^2 + 2q(x)$
would be hyperbolic, but not SOS-hyperbolic, with respect to $e_0$.  

\section{Canonical linear functionals and the proof of Theorem~\ref{thm:main1}}
\label{sec:pf}
This section is devoted to developing tools used in the proof of Theorem~\ref{thm:main1} (and in Section~\ref{sec:dual} to follow), 
and then presenting the proof of Theorem~\ref{thm:main1} itself.  We begin by showing how a hyperbolic polynomial
$p\in \Hyp_{n,d}(e)$, and a point $x\in \RR^n$, define a set of linear functionals that generalize the orthogonal projectors
onto the eigenspaces of a symmetric matrix. This construction may be of independent interest.

\subsection{Canonical linear functionals}
\label{sec:clf}
Given $p\in \Hyp_{n,d}(e)$ and a fixed $x\in \RR^n$, we define a collection of
linear functionals associated with the eigenvalues of $x$. These are a
generalization of the orthogonal projectors onto the eigenspaces of a symmetric
matrix, as we shall see in Example~\ref{eg:deteg}. If $x$ has simple
eigenvalues, these canonical linear functionals are the directional derivatives
of the eigenvalues.  If $u\in \RR^n$, recall that 
\[ p(x+te+su) = p(e)\prod_{i=1}^{d}(t+t_i(s;x,u))\]
where the $t_i(s;x,u)$ are defined in Theorem~\ref{thm:abg}. Let $x$ have
eigenvalues $\lambda_1(x)> \ldots > \lambda_k(x)$ with multiplicities
$m_1,\ldots,m_k$. Let $I_1,\ldots,I_k$ be a partition of $\{1,2,\ldots,d\}$
such that $|I_i| = m_i$ and $t_j(0;x,u) = \lambda_i(x)$ for all $j\in I_i$.
Define the \emph{canonical linear functionals} to be
\[ \lambda_i'(x)[u]:= \sum_{j\in I_i}t_j'(0;x,u)\quad\textup{for $i=1,2,\ldots,k$}.\]

The following result summarizes the main properties of the $\lambda_i'(x)[u]$ that we will need.
\begin{lemma}
\label{lem:clf-prop}
If $p\in \Hyp_{n,d}(e)$ and $x\in \RR^n$, then the maps $\RR^n \ni u\mapsto \lambda_i'(x)[u]$ satisfy:
\begin{enumerate}
	\item $\displaystyle{\frac{D_up_x(t)}{p_x(t)} = \sum_{i=1}^{k}\frac{\lambda_i'(x)[u]}{\lambda_i(x)+t}}$;
	\item $u\mapsto \lambda_i'(x)[u]$ is linear;
	\item $\lambda_i'(x)[e] = m_i$, the multiplicity of the $i$th eigenvalue of $x$;
	\item $\lambda_i'(x)[x] = m_i\lambda_i(x)$; 
	\item $u\in \Lambda_+(p,e)$ if, and only if, $\lambda_i'(x)[u] \geq 0$ for all $x\in \RR^n$ and all $i$.
\end{enumerate}
\end{lemma}
\begin{proof}
The first statement follows from the following straightforward computation
\begin{equation*}
	 \frac{D_up_x(t)}{p_x(t)} = \frac{d}{ds}\left.\log(p(x+te+su))\right|_{s=0}
		 = \sum_{j=1}^{d}\frac{t_j'(0;x,u)}{t_j(0;x,u) + t}
		 = \sum_{i=1}^{k}\frac{\lambda'_i(x)[u]}{\lambda_i(x)+t}.
\end{equation*}
For the second statement, note that the numerator of $\frac{D_up_x(t)}{p_x(t)}$ is linear in $u$, so 
the numerators, $\lambda_i'(x)[u]$, of the partial fraction expansion in the first statement are also linear in $u$.

The third statement follows from the fact that $t_j(s;x,e) = \lambda_i(x) + s$ for all $i\in I_i$ so that $\lambda_i'(x)[e] = |I_i| = m_i$. The fourth statement follows from the fact that $t_j(s;x,x) = \lambda_i(x)(1+s)$ for all $i\in I_i$ so that $\lambda'_i(x)[x] = |I_i|\lambda_i(x) = m_i\lambda_i(x)$. 

For the final statement, on the one hand, if $u\in \Lambda_+(p,e)$, then $t_j'(0;x,u) \geq 0$ (by Theorem~\ref{thm:abg}), 
and so $\lambda_i'(x)[u] \geq 0$.
On the other hand, if $\lambda'_i(x)[u] \geq 0$ for all $x\in \RR^n$ and all $i$, then $\frac{1}{m_i}\lambda'_i(u)[u] = \lambda_i(u) \geq 0$ for all $i$,  and so $u\in \Lambda_+(p,e)$. 
\end{proof}
\begin{remark}
We could have defined $\lambda_i'(x)[u]$ via the first property in Lemma~\ref{lem:clf-prop}, as the residue of 
$D_up_x(t)/p_x(t)$ corresponding to the simple pole at $-\lambda_i(x)$. Then we would need to show, directly, 
that these are nonnegative if, and only if,  $u\in \Lambda_+(p,e)$, rather than relying on Theorem~\ref{thm:abg}.  
\end{remark}

\begin{example}[Canonical linear functionals for the determinant]
\label{eg:deteg}
Let $p(X) = \det(X)$ where $X$ is a symmetric matrix and $e=I$. 
Suppose $X$ has $k$ distinct eigenvalues $\lambda_1(X),\ldots,\lambda_k(X)$ with corresponding multiplicities
$m_1,\ldots,m_k$. Then  
$X = \sum_{i=1}^{k}\lambda_i(X) P_i(X)$
where $P_i(X)$ is the orthogonal projector onto the eigenspace of $X$ corresponding to $\lambda_i(X)$. The associated canonical 
linear functionals are the maps $U \mapsto \tr(P_i(X) U)$. To see why, note that 
\[ \frac{D_U\det(X+tI)}{\det(X+tI)} = \tr(U(X+tI)^{-1}) = \sum_{i=1}^{k}\frac{\tr(UP_i(X))}{t+\lambda_i(X)}.\]
These satisfy $\tr(P_i(X) I) = m_i$ and 
$\tr(P_i(X) X) = m_i\lambda_i(X)$ and $\tr(P_i(X) U)\geq 0$ for all $i$ and all $X$ if, and only if, $U\psd 0$. 
\end{example}

\subsection{Proof of Theorem~\ref{thm:main1}}
By Proposition~\ref{prop:congx}, it is enough to show that $u\in \Lambda_+(p,e)$ if, and only if, $H_{p,e}(x)[u] \psd 0$ for all $x$.
To achieve this, our main task is to 
express the parameterized Hermite matrix in terms of the canonical linear functionals.

\begin{lemma}
\label{lem:hermite-clf}
Let $p\in \Hyp_{n,d}(e)$ and $u\in \RR^n$. With $y\in \RR^d$ define a univariate polynomial
$q_y(t) = y_1+y_2t + \cdots + y_dt^{d-1}$.
Then, using the notation of Section~\ref{sec:clf}, we have that
\[ y^TH_{p,e}(x)[u]y = \sum_{i=1}^{k}[q_y(-\lambda_i(x))]^2\lambda_i'(x)[u].\]
\end{lemma}

\begin{proof}
We expand the rational function $t\mapsto \frac{D_up_x(t)}{p_x(t)}$ as a Laurent series via
\[ \frac{D_up_x(t)}{p_x(t)} = \sum_{i=1}^{k}\frac{\lambda'_i(x)[u]}{t+\lambda_i(x)} = \sum_{i=1}^{k}
\lambda_i'(x)[u]\sum_{\ell=1}^{\infty}
(-\lambda_i(x))^{\ell-1}t^{-\ell}.\]
It then follows that  
$[H_{p,e}(x)[u]]_{j,\ell} = \sum_{i=1}^{k}\lambda_i'(x)[u](-\lambda_i(x))^{j+\ell-2}$.
The result follows because, for any $1\leq i \leq k$,  
$\sum_{j,\ell=1}^{d}y_jy_\ell(-\lambda_i(x))^{j+\ell-2} = [q_y(-\lambda_i(x))]^2$.
\end{proof}
 Specialized to the determinant, this agrees with the explicit computation from Example~\ref{eg:hermdet}. 
 \begin{example}[Hankel matrix for the determinant]
 \label{eg:hermdeth}
 Using the notation of Example~\ref{eg:deteg}, the quadratic form defined by the Hankel matrix is
 \begin{align*}
	 y^TH_{p,e}(X)[U]y & = \textstyle{\sum_{i=1}^{k}q_y(-\lambda_i(X))^2 \tr(P_i(X)U)}\\
			& = \textstyle{\tr\left[\left(\sum_{i=1}^{k}q_y(-\lambda_i(X))^2P_i(X)\right)U\right]} = \tr((q_y(-X))^2U)
\end{align*}
 where $q_y(-X) = y_1 + y_2(-X) + y_3(-X)^2 + \cdots + y_{d}(-X)^{d-1}$. 
 \end{example}
We are now in a position to complete the proof of Theorem~\ref{thm:main1}.
\begin{corollary}
If $p\in \Hyp_{n,d}(e)$, then $u\in \Lambda_+(p,e)$ if, and only if, $H_{p,e}(x)[u] \psd 0$ for all $x$. 
\end{corollary}
\begin{proof}
If $u\in \Lambda_+(p,e)$, then $\lambda'(x)[u]\geq 0$ for all $i$ and all $x$. From Lemma~\ref{lem:hermite-clf}
it follows that $y^TH_{p,e}(x)[u]y \geq 0$ for all $x$ and all $y$. Conversely, suppose that $\lambda_i(u) < 0$ for some $i$. 
Define $y$ so that the polynomial $q_y$ satisfies $q_y(-\lambda_j(x)) = 0$ if $j\neq i$ and $q_y(-\lambda_j(x))=1$ if $j=i$. 
Such a polynomial can be constructed via Lagrange interpolation, for instance. Then 
\[ y^TH_{p,e}(u)[u]y = \lambda_i'(u)[u] = m_i\lambda_i(u) < 0.\]
\end{proof}

\section{View from the dual cone}
\label{sec:dual}

In this section we consider the results obtained so far from the point of view of the \emph{dual cone}
\[ \Lambda_+(p,e)^* := \{ \xi\in (\RR^{n})^*\;:\; \xi[x] \geq 0\;\;\textup{for all $x\in \Lambda_+(p,e)$}\},\]
the (closed convex) cone of linear functionals that take nonnegative values on $\Lambda_+(p,e)$. 
In particular, we study the image of the polynomial map $\phi_{p,e} = \phi^H_{p,e}$, defined in~\eqref{eq:phiH}.
We will see that this image contains the interior of the dual cone and is contained in the dual cone.
Moreover, we give examples in which the image of $\phi_{p,e}$
is precisely the dual cone $\Lambda_+(p,e)^*$. 

The image of $\phi_{p,e}$ can be expressed using the canonical linear functionals of Section~\ref{sec:clf}.
\begin{proposition}
\label{prop:image-cone}
If $p\in \Hyp_{n,d}(e)$ and $x\in \RR^n$ has $k\leq d$ distinct hyperbolic eigenvalues, then 
\[ \phi_{p,e}(x,\RR^d) = \textup{cone}\{\lambda_1'(x)[\cdot],\ldots,\lambda_k'(x)[\cdot]\}.\]
\end{proposition}
\begin{proof}
From Lemma~\ref{lem:hermite-clf} we know that 
$\phi_{p,e}(x,y) = \sum_{i=1}^{k}(q_y(-\lambda_i(x)))^2\lambda'_i(x)[\cdot]$
which is clearly a nonnegative combination of the linear functionals $\lambda_i'(x)[\cdot]$. 
Conversely, given any nonnegative scalars $\eta_1,\eta_2,\ldots,\eta_k$, we can 
choose $y\in \RR^d$ so that the polynomial $q_y$ (of degreee $d-1$)
satisfies $q_y(-\lambda_i(x)) = \sqrt{\eta_i}$ (by interpolation). This shows that any nonnegative combination 
of the $\lambda_i'(x)[\cdot]$ is in the set $\phi_{p,e}(x,\RR^d)$. 
\end{proof}
\subsection{Parameterization of the dual cone up to closure}

The following result shows that the map $\phi_{p,e}$ almost parameterizes the dual of the hyperbolicity cone. 
One inclusion follows directly from Theorem~\ref{thm:main1}. The other direction follows, for instance, from 
the (well-known) fact that the gradient of $\log p(x)$ maps the interior of the hyperbolicity cone
onto the relative interior of its dual cone. 
In the statement of the theorem we use the notation $\textup{int}(S)$ to denote the interior of the set $S$ and $\textup{relint}(S)$
to denote the relative interior. 
\begin{theorem}
\label{thm:dualma}
If $p\in \Hyp_{n,d}(e)$, then 
$\textup{relint}(\Lambda_+(p,e)^*) \subseteq \phi_{p,e}(\RR^n,\RR^d) \subseteq \Lambda_+(p,e)^*$.
\end{theorem}
\begin{proof}
The right hand inclusion follows from Theorem~\ref{thm:main1}. Indeed if $u\in \Lambda_+(p,e)$, then 
\[ \phi_{p,e}(x,y)[u] \geq 0\quad\textup{for all $x\in \RR^n$ and all $y\in \RR^d$}\]
which implies that $\phi_{p,e}(\RR^n,\RR^d)\subseteq \Lambda_+(p,e)^*$. For the left-hand  inclusion,
let $\xi$ be an arbitrary linear functional in the relative interior of $\Lambda_+(p,e)^*$. 
Let $x$ be any
optimal point (which exists because $\xi$ is in the relative interior of the dual cone) 
of the convex optimization problem 
\[ \min_{x\in \Lambda_+(p,e)} -\log p(x)\;\;\textup{subject to}\;\;\xi[x]=d.\]
Then $x\in  \textup{int}(\Lambda_+(p,e))$ (because the objective function
goes to infinity on the boundary) and so $\lambda_i(x)>0$ for all $i$. From the optimality conditions
\[ \xi[u] = D_u \log p(x) = \sum_{i=1}^{k} \frac{\lambda_i'(x)[u]}{\lambda_i(x)}\quad\textup{for all $u\in \RR^n$.}\]
It follows that $\xi[\cdot] \in \textup{cone}\{\lambda_1'(x)[\cdot],\ldots,\lambda_k'(x)[\cdot]\}\subseteq \phi_{p,e}(x,\RR^d)$ 
(by Proposition~\ref{prop:image-cone}).
\end{proof}

\begin{remark}
If $p$ is a complete hyperbolic polynomial, the gradient of $\log p(x)$ gives a rational bijection
between the interior of a hyperbolicity cone and the interior of its dual cone.
A key difference between $\phi_{p,e}$ and the gradient of $\log p(x)$ 
is that the former maps \emph{all} of $\RR^n\times \RR^d$ into the dual cone $\Lambda_+(p,e)^*$, where as the latter 
only maps the hyperbolicity cone itself into the dual cone. This difference allows us to construct polynomials
that are globally nonnegative, rather than just nonnegative restricted to the open hyperbolicity cone. 
\end{remark}

We have seen that the image of $\phi_{p,e}$ is sandwiched between the interior of the dual cone and the dual cone itself. 
The following example shows that the image of $\phi_{p,e}$ may not contain all of the boundary of the dual cone.

\begin{example}[{Singular cubic curve}]
\label{eg:scubic}
Consider the hyperbolic cubic
\[ p(x_1,x_2,x_3) = \det\cA(x)\quad\textup{where}\quad
\cA(x) = \begin{bmatrix} x_3 & 0 & x_1\\0 & x_1+x_3 & x_2\\x_1 & x_2 & x_3\end{bmatrix}\]
taken from~\cite[Section 4.1]{henrion2010semidefinite}. From the definite determinantal representation 
we can see that $p$ is hyperbolic with respect to $e = (0,0,1)$. Its dual cone can be parameterized by
\[ \Lambda_+(p,e)^* = \{\cA^*(zz^T) = (z_2^2 + 2z_1z_3, 2z_2z_3, z_1^2 + z_2^2 + z_3^2): z\in \RR^3\}\]
as discussed in~\cite{henrion2010semidefinite}. The intersection of the hyperbolicity cone with the affine hyperplane
$\{x: \langle e,x\rangle = 1\}$ and the intersection of 
$\phi_{p,e}(\RR^3,\RR^3)$ with the affine hyperplane 
$\{z:\langle e,z\rangle = 1\}$ are shown in Figure~\ref{fig:scubic}.
\begin{figure}
\begin{center}
\includegraphics[scale=1]{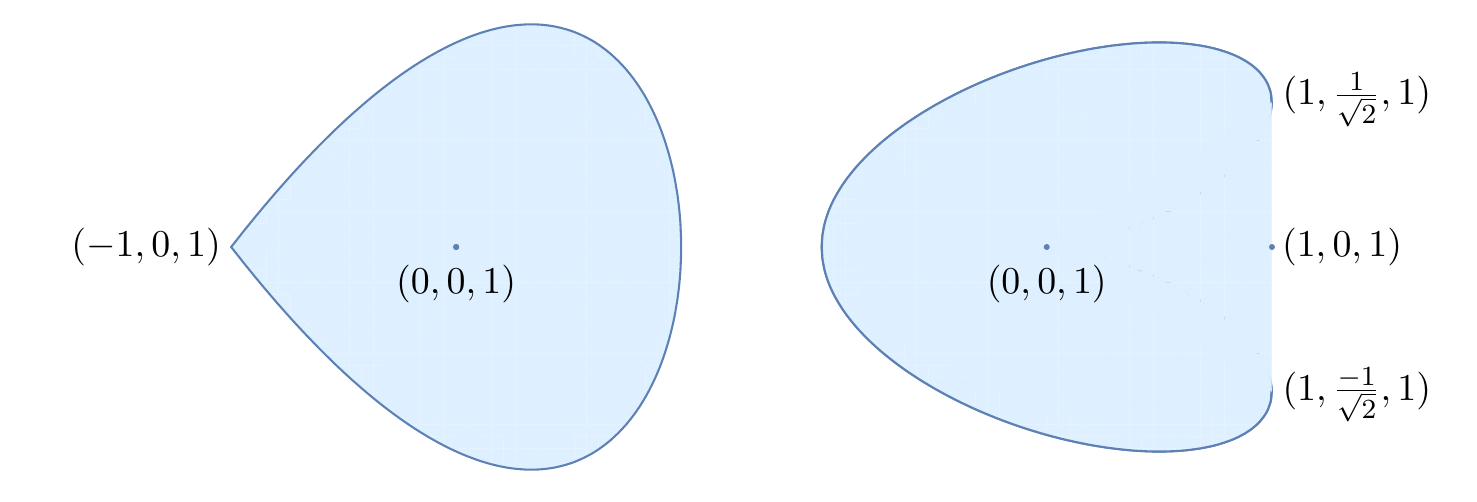}
\end{center}
\caption{\label{fig:scubic} On the left is $\Lambda_+(p,e)\cap \{x:\langle e,x\rangle=1\}$, and on the right is $\phi_{p,e}(\RR^3,\RR^3) \cap \{z: \langle z,e\rangle = 1\}$, where $p$ is the cubic from Example~\ref{eg:scubic}.}
\end{figure}
The intersection of $\Lambda_+(p,e)^*$ with the
same affine hyperplane is the closure of the set shown on the right in Figure~\ref{fig:scubic}.
The dual cone $\Lambda_+(p,e)^*$ and the image of $\phi_{p,e}$ differ in  one has a 
two-dimensional face given by the conic hull of $(1,1/\sqrt{2},1)$ and $(1,-1/\sqrt{2},1)$, and in the other
this face is replaced by the single ray generated by $(1,0,1)$. 
\end{example}

\subsection{Exact parameterizations of the dual cone}

We now study certain situations in which the image of $\phi_{p,e}$ is precisely the closed dual cone. 
From Example~\ref{eg:scubic}, we know that this does not always occur. 
In this section we show that the image of $\phi_{p,e}$ is the closed dual cone $\Lambda_+(p,e)^*$
when, for instance, $p$ is strictly hyperbolic (see below for a definition), $p$ is a product of linear forms, 
or $p$ is the determinant restricted to symmetric matrices. 

If $p\in \Hyp_{n,d}(e)$ and $x\in \RR^n$, we define the \emph{multiplicity} of $x$ to be 
the multiplicity of the root $t=0$ of $p_x(t)$. If $p_x(0) \neq 0$, then we define the multiplicity of 
$x$ to be zero. A polynomial $p\in \Hyp_{n,d}(e)$ is \emph{strictly hyperbolic} if every element 
$x\in \RR^n$ that is not a multiple of $e$ has $n$ distinct hyperbolic eigenvalues. It follows 
that if $p$ is strictly hyperbolic and $x$ is nonzero and satisfies $p(x) = 0$, then $x$ has 
multiplicity $1$, the gradient of $p$ at $x$ is non-zero, and $x$ is a smooth point of the boundary of 
the hyperbolicity cone.  

\begin{proposition}
If $p\in \Hyp_{n,d}(e)$ is strictly hyperbolic, then $\phi_{p,e}(\RR^n,\RR^d) = \Lambda_+(p,e)^*$. 
\end{proposition}
\begin{proof}
From Theorem~\ref{thm:dualma} we know that 
$\Lambda_+(p,e)^*\supseteq \phi_{p,e}(\RR^n,\RR^d) \supseteq \textup{int}(\Lambda_+(p,e)^*)$. 
It suffices to show that for any $\xi\in \partial \Lambda_+(p,e)^*$ there exists $x\in \RR^n$ such 
that $\xi\in \phi_{p,e}(x,\RR^d)$. 

Any non-zero $\xi\in \partial \Lambda_+(p,e)^*$ must vanish at some $x_\xi\in \partial\Lambda_+(p,e)\setminus\{0\}$ 
(since otherwise $\xi$ would be in the interior of the dual cone). Since $x_\xi$ is a non-zero element of the boundary
of the hyperbolicity cone, it is a smooth point. As such, there is a unique supporting hyperplane to the cone at $x_\xi$. 
Since $x$ is in the boundary of the hyperbolicity cone, $\lambda_{\min}(x_\xi) = 0$ and so $\lambda_{\min}'(x_\xi)[x_\xi] = 0$. 
Both $\xi$ and $\lambda_{\min}'(x_\xi)$ define supporting hyperplanes at $x_\xi$, so one must be a nonnegative 
multiple of the other. In particular $\xi$ is in the cone over the canonical linear functionals at $x_\xi$ and so 
is an element of the image of $\phi_{p,e}$. 
\end{proof}
In the case where our hyperbolic polynomial is a product of distinct linear forms, 
so that the resulting hyperbolicity cone is polyhedral,
the image of $\phi_{p,e}$ is the closed dual cone.
\begin{proposition}
If $p\in \Hyp_{n,d}(e)$ is a product of distinct linear forms, i.e., $p(x) = \prod_{i=1}^{d}a_i[x]$  with $a_i[e]>0$ for all $i$ and $a_i\neq a_j$ for all $i,j$, then $\phi_{p,e}(\RR^n,\RR^d) = \Lambda_+(p,e)^*$. 
\end{proposition}
\begin{proof}
Let $x$ be a generic point in $\RR^n$, so that $x$ has $d$ distinct eigenvalues which are 
$\lambda_i(x) = a_i[x]/a_i[e]$ for $i=1,2,\ldots,d$. Since the eigenvalue functions are linear, 
the canonical linear functions are simply $\lambda_i'(x)[u] = a_i[u]/a_i[e]$ for $i=1,2,\ldots,d$.
From Proposition~\ref{prop:image-cone} we know that 
\[ \phi_{p,e}(x,\RR^d) = \textup{cone}\{\lambda_1'(x),\ldots,\lambda_d'(x)\}.\]
This is the cone generated by the linear forms $a_1,\ldots,a_d$, which is just  $\Lambda_+(p,e)^*$. 
\end{proof}
Another case in which the image of $\phi_{p,e}$ is the dual cone is when $p$ is the determinant restricted to symmetric matrices.
\begin{proposition}
Let $p(X) = \det(X)$ be the determinant restricted to $d\times d$ real
symmetric matrices. Then $\phi_{p,I}(\cS^d,\RR^d) = \Lambda_+(p,I)^*$.
\end{proposition}
\begin{proof}
Since the positive semidefinite cone is self-dual it follows that an arbitrary element of $\Lambda_+(p,I)^*$ 
has the form $U \mapsto \tr(ZU)$ where $Z\psd 0$. 
Let $y = e_2 = (0,1,0,\ldots,0)$ so that $q_y(t) = t$ in the notation of Example~\ref{eg:hermdeth}. 
Let $X = Z^{1/2}$ be the positive semidefinite square root of $Z$. 
If $U\in \cS^d$, from Example~\ref{eg:hermdeth} we have that 
\[ \phi_{p,I}(X,y) = \tr((-X)^2U) = \tr(X^2U) = \tr(ZU).\]
Since $Z$ was an arbitrary positive semidefinite matrix,  we are done.
\end{proof}

\section{Connections with interlacers}
\label{sec:discussion}

To conclude, we discuss the connection between our work and interlacing polynomials. 
If $p\in \Hyp_{n,d}(e)$ and $q\in \Hyp_{n,d-1}(e)$, we say that \emph{$q$ interlaces $p$ with respect to $e$}
if the hyperbolic eigenvalues of any $x\in \RR^n$ with respect to $p$ and $q$ satisfy
\[ \lambda_1^p(x)\geq \lambda_1^q(x) \geq \lambda_2^p(x) \geq \cdots \geq \lambda_{d-1}^q(x) \geq \lambda_d^p(x).\]
For a fixed hyperbolic polynomial $p\in \Hyp_{n,d}(e)$ the cone $\textup{Int}_e(p)$
of polynomials $q$ that interlace $p$ with respect to $e$ is a convex cone.
In~\cite{kummer2015hyperbolic,kummer2018spectrahedral} it is shown that the cone of interlacers of $p\in \Hyp_{n,d}(e)$ 
can be described by 
\[ \textup{Int}_e(p) = \{q\in \RR[x]_{d-1}\;:\; B(p_x,q_x) \psd 0\quad\textup{for all $x\in \RR^n$}\}.\]
Much of the development of Sections~\ref{sec:cert} and~\ref{sec:pf} could have been presented
from the point of view of interlacers. This is because $D_up(x)\in
\textup{Int}_e(p)$ if, and only if, $u\in \Lambda_+(p,e)$, so 
the hyperbolicity cone is a section of the cone of interlacers~\cite{kummer2015hyperbolic}. 

The argument we used to show that every ternary hyperbolic polynomial is SOS-hyperbolic actually 
generalizes to give a projected spectrahedral description of the cone of interlacers of 
any ternary hyperbolic polynomial. This was first observed by Kummer, Naldi, and Plaumann~\cite{kummer2018spectrahedral}
via a seemingly quite different proof.
\begin{proposition}[{\cite[Page 17]{kummer2018spectrahedral}}]
If $p\in \Hyp_{3,d}(e)$, then 
\begin{align*}
 \textup{Int}_e(p) = \{q\in \RR[x_1,x_2,x_3]_{d-1}: B(p_x,q_x)\;\; \textup{is a matrix sum of squares}\}.
\end{align*}
\end{proposition}
\begin{proof}
Recall that $B(p_x,q_x)$ is a matrix sum of squares if, and only if, the restriction to a two-dimensional subspace
not containing $e$ is a matrix sum of squares. But, after choosing appropriate coordinates, this is a matrix with 
entries that are each homogeneous forms in two variables, and so is a matrix sum of squares 
via~\cite[Remark 5.10]{blekherman2016sums}.
\end{proof}
If $p$ is hyperbolic and has degree two, then the cone of interlacers of $p$ with respect to $e$ is simply the 
set of linear forms that are nonnegative on the cone $\Lambda_+(p,e)$. This is the dual cone,
$\Lambda_+(p,e)^*$, which is again a quadratic cone, and so is a spectrahedron. Another case 
where we might expect the cone of interlacers to have a nice description is the case of 
hyperbolic polynomials of degree three in four variables.
\begin{question}
If $p\in \Hyp_{4,3}(e)$, is the cone of interlacers $\textup{Int}_e(p)$ a projected spectrahedron?
\end{question}

\paragraph{Acknowledgements}
I would like to thank Amir Ali Ahmadi, Petter Br\"and\'en, Hamza Fawzi, Mario Kummer, Simone Naldi, Pablo Parrilo, 
Levent Tun{\c{c}}el, and Cynthia Vinzant for helpful discussions and correspondence related to various aspects of this work, 
and the anonymous referees for their thoughtful suggestions that have improved the article.

\bibliographystyle{alpha}
\bibliography{NNHYP-bib}

\appendix
\section{Relating B\'ezoutians and Hankel matrices}
\label{app:prelim}
In this appendix we prove Proposition~\ref{prop:cong-uni}. Before doing so, we establish a slightly simpler statement
in which the dimension of the B\'ezoutian and Hankel matrices is the same as $\textup{deg}(a)$.
 \begin{proposition}
\label{prop:conj1}
 If $a(t) = a_0 + a_1t + \cdots + t^d$ is a monic polynomial of degree $d$ 
 and $b$ is a polynomial of degree strictly less than $d$, then 
 \begin{equation*}
 	 B_d(a,b) = B(a,1)H_d\!\left(\frac{b}{a}\right)B(a,1)\quad\textup{where}\quad
		[B(a,1)]_{ij} = \begin{cases} a_{i+j-1} & \textup{if $1\leq i+j-1\leq d-1$}\\
 		1 & \textup{if $i+j-1 = d$}\\
 		0 & \textup{otherwise}\end{cases}
 \end{equation*}
 is the (symmetric and unimodular) B\'ezoutian of $a$ and the constant polynomial $1$.
 \end{proposition}
\begin{proof}
This result is established in~\cite[Proposition 7.15]{bini2011numerical}. We reproduce the proof here because there are 
some confusing typographic errors in~\cite[page 75]{bini2011numerical}. 
The argument presented there uses the fact that the Hankel matrix $H_d\!\left(\frac{b}{a}\right)$ satisfies the identity
\[ \frac{\frac{b(t)}{a(t)} - \frac{b(s)}{a(s)}}{s-t} = \sum_{p,q=1}^{\infty}h_{p+q-1}t^{-p}s^{-q}.\]
Combining this with the defining identity of the B\'ezoutian~\eqref{eq:beqid} we obtain 
\begin{equation*}
\sum_{i,j=1}^{d} [B(a,b)]_{ij}t^{i-1}s^{j-1} 
= a(t)\frac{\frac{b(t)}{a(t)} - \frac{b(s)}{a(s)}}{s-t}a(s) = 
\sum_{m,\ell=0}^{d}\sum_{p,q=1}^{\infty}a_{m}h_{p+q-1}a_\ell t^{m-p}s^{\ell-p}.
\end{equation*}
Changing variables via $i= m-p+1$ and $j = \ell -q+1$, and defining $a_i = 0$ for $i>d$, gives
\[ \sum_{i,j=1}^{d} [B(a,b)]_{ij}t^{i-1}s^{j-1} = \sum_{i,j=1}^{d}\left[\sum_{p,q=1}^{\infty}a_{i+p-1}h_{p+q-1}a_{j+q-a}\right]t^{i-1}s^{j-1}\]
from which the matrix identity follows.

The fact that $B(a,1)$ is unimodular follows directly from the fact that, if we reverse the order of the rows of $B(a,1)$ (which 
either preserves the determinant or changes its sign), we obtain a lower triangular matrix with unit diagonal. 
\end{proof}
We now consider the case in which the dimension of the B\'ezoutian and Hankel matrices is possibly larger than the degree of the 
polynomial $a$.
\begin{proof}[{Proof of Proposition~\ref{prop:cong-uni}}]
Let $\tilde{a}(t) = t^{m-d}a(t)$ and $\tilde{b}(t) = t^{m-d}b(t)$. First note that 
\begin{align*}
	B_m(\tilde{a},\tilde{b}) = \begin{bmatrix} 0_{m-d,m-d} & 0_{m-d,d}\\0_{d,m-d} & B_d(a,b)\end{bmatrix} & \! =\! 
	\begin{bmatrix} 0_{m-d,d} & I_{m-d}\\I_{d} & 0_{d,m-d}\end{bmatrix} \!\!
	\begin{bmatrix} B_d(a,b) & 0_{d,m-d}\\0_{m-d,d} & 0_{m-d,m-d}\end{bmatrix}\!\!
	\begin{bmatrix} 0_{m-d,d} & I_{m-d}\\I_{d} & 0_{d,m-d}\end{bmatrix}^T\\
&\! = \! 
	\begin{bmatrix} 0_{m-d,d} & I_{m-d}\\I_{d} & 0_{d,m-d}\end{bmatrix}B_m(a,b) 
	\begin{bmatrix} 0_{m-d,d} & I_{m-d}\\I_{d} & 0_{d,m-d}\end{bmatrix}^T.
\end{align*}
By Proposition~\ref{prop:conj1},
and the fact that $\frac{\tilde{b}}{\tilde{a}} = \frac{b}{a}$, we see that
\begin{align*}
	 B(\tilde{a},1)H_m\!\left(\frac{\tilde{b}}{\tilde{a}}\right)B(\tilde{a},1) &= 
B(\tilde{a},1)H_m\!\left(\frac{b}{a}\right)B(\tilde{a},1)\\ & =  B_m(\tilde{a},\tilde{b}) =
\begin{bmatrix} 0_{d,m-d} & I_{d}\\I_{m-d} & 0_{m-d,d}\end{bmatrix} B_m(a,b)\begin{bmatrix} 0_{d,m-d} & I_{d}\\I_{m-d} & 0_{m-d,d}\end{bmatrix}^T.
\end{align*}
Since $B(\tilde{a},1)$ is symmetric and unimodular and has entries that are linear in the coefficients of $a$, we see that 
\[ M(a) = \begin{bmatrix} 0_{d,m-d} & I_{d}\\I_{m-d} & 0_{m-d,d}\end{bmatrix}^{-1}B(\tilde{a},1)\]
is unimodular and has entries that are linear in the coefficients of $a$, completing the proof.
\end{proof}

\end{document}